\documentclass[reqno,a4paper]{amsart}
\usepackage{amssymb}
\usepackage{amsmath}
\newcommand{\angles}[1]{\langle #1 \rangle}
\newcommand{\half}{\frac{1}{2}}
\newcommand{\abs}[1]{\vert #1 \vert}
\newcommand{\norm}[1]{\left\Vert #1 \right\Vert}
\newcommand{\R}{\mathbb{R}}

\begin{document} 
\newtheorem{prop}{Proposition}[section]
\newtheorem{Def}{Definition}[section]
\newtheorem{theorem}{Theorem}[section]
\newtheorem{lemma}{Lemma}[section]
 \newtheorem{Cor}{Corollary}[section]

\title[Yang-Mills in Lorenz gauge]{\bf Low regularity well-posedness for the Yang-Mills system in 2D}
\author[Hartmut Pecher]{
{\bf Hartmut Pecher}\\
Fakult\"at f\"ur  Mathematik und Naturwissenschaften\\
Bergische Universit\"at Wuppertal\\
Gau{\ss}str.  20\\
42119 Wuppertal\\
Germany\\
e-mail {\tt pecher@math.uni-wuppertal.de}}
\date{}

\begin{abstract}
The Cauchy problem for the Yang-Mills system in two space dimensions is treated for data with minimal regularity assumptions. In the classical case of data in $L^2$-based Sobolev spaces we have to assume that the number of derivatives is more than $3/4$  above the critical regularity with respect to scaling. For data in $L^r$-based Fourier-Lebesgue spaces this result can be improved by $1/4$ derivative in the sense of scaling as $r \to 1$ .
\end{abstract}
\maketitle
\renewcommand{\thefootnote}{\fnsymbol{footnote}}
\footnotetext{\hspace{-1.5em}{\it 2010 Mathematics Subject Classification:} 
35Q40, 35L70 \\
{\it Key words and phrases:} Yang-Mills,  
local well-posedness, Lorenz gauge}
\normalsize 
\setcounter{section}{0}

\section{Introduction}

\noindent 
Let $\mathcal{G}$ be the Lie group $SO(n,\mathbb{R})$ (the group of orthogonal matrices of determinant 1) or $SU(n,\mathbb{C})$ (the group of unitary matrices of determinant 1) and $g$ its Lie algebra $so(n,\mathbb{R})$ (the algebra of trace-free skew symmetric matrices) or $su(n,\mathbb{C})$ (the algebra of trace-free skew hermitian matrices) with Lie bracket $[X,Y] = XY-YX$ (the matrix commutator). 
For given  $A_{\alpha}: \mathbb{R}^{1+n} \rightarrow g $ we define the curvature $F=F[A]$ by
\begin{equation}
\label{curv}
 F_{\alpha \beta} = \partial_{\alpha} A_{\beta} - \partial_{\beta} A_{\alpha} + [A_{\alpha},A_{\beta}] \, , 
\end{equation}
where $\alpha,\beta \in \{0,1,...,n\}$ and $D_{\alpha} = \partial_{\alpha} + [A_{\alpha}, \cdot \,]$ .

Then the Yang-Mills system is given by
\begin{equation}
\label{1}
D^{\alpha} F_{\alpha \beta}  = 0
\end{equation}
in Minkowski space $\mathbb{R}^{1+n} = \mathbb{R}_t \times \mathbb{R}^n_x$ , where $n \ge 2$, with metric $diag(-1,1,...,1)$. Greek indices run over $\{0,1,...,n\}$, Latin indices over $\{1,...,n\}$, and the usual summation convention is used.  
We use the notation $\partial_{\mu} = \frac{\partial}{\partial x_{\mu}}$, where we write $(x^0,x^1,...,x^n)=(t,x^1,...,x^n)$ and also $\partial_0 = \partial_t$.

Setting $\beta =0$ in (\ref{1}) we obtain the Gauss-law constraint
\begin{equation}
\nonumber
\partial^j F_{j 0} + [A^j,F_{j0} ]=0 \, .
\end{equation}

The total energy for Yang-Mills at time $t$  is given by
$$
  \mathcal E(t) =\sum_{0\le \alpha, \beta\le n} \int_{\R^n} \abs{F_{\alpha \beta}(t,x)}^2  \, dx,
$$
and is conserved for a smooth solution decaying sufficiently fast at spatial infinity. 
The  Yang-Mills system is invariant with respect to the scaling
$$ A_{\lambda}(t,x) = \lambda A(\lambda t,\lambda x) \quad, \quad F_{\lambda}(t,x)= \lambda^2 F(\lambda t,\lambda x) \, . $$
This implies
\begin{align*}
\|A_{\lambda}(0,\cdot)\|_{\dot{\widehat{H}}^{s,r}} = \lambda^{1+s-\frac{n}{r}} \|a_{\lambda}\|_{\dot{\widehat{H}}^{s,r}} \, , \\
\|F_{\lambda}(0,\cdot)\|_{\dot{\widehat{H}}^{l,r}} = \lambda^{2+l-\frac{n}{r}} \|f_{\lambda}\|_{\dot{\widehat{H}}^{l,r}}	\, .
\end{align*}
Here $\|u\|_{\widehat{H}^{s,r}} := \| \langle \xi \rangle^s \widehat{u}(\xi)\|_{L^{r'}}$ , where $r$ and $r'$ are dual exponents, and $\hat{\widehat{H}}^{s,r}$ denotes the homogeneous space.
Therefore the scaling critical exponent is $s = \frac{n}{r}-1$ for $A$ and $l= \frac{n}{r}-2$ for $F$ .

The system is gauge invariant. Given a sufficiently smooth function $U: {\mathbb R}^{1+n} \rightarrow \mathcal{G}$ we define the gauge transformation $T$ by $T A_0 = A_0'$ , 
$T(A_1,...,A_n) = (A_1',...,A_n'),$ where
\begin{align*}
A_{\alpha} & \longmapsto A_{\alpha}' = U A_{\alpha} U^{-1} - (\partial_{\alpha} U) U^{-1}  \, . 
\end{align*}

It is  well-known that if  $(A_0,...A_n)$ satisfies (\ref{curv}),(\ref{1}) so does $(A_0',...,A_n')$.

Hence we may impose a gauge condition. We exclusively study the case $n=2$ and Lorenz gauge $\partial^{\alpha}A_{\alpha} =0$. Other convenient gauges are the Coulomb gauge $\partial^j A_j=0$ and the temporal gauge $A_0 =0$.
Our aim is to obtain local well-posedness for data with minimal regularity. 

The classical case $n=3$ and $r=2$ with data in standard Sobolev spaces was considered by Klainerman and Machedon \cite{KM}, who made the decisive detection that
the nonlinearity satisfies a so-called null condition, which enabled them to prove global well-posedness in temporal and in Coulomb gauge in energy space. The corresponding result in Lorenz gauge, where the Yang-Mills equations can be formulated as a system of nonlinear wave equations, was shown by Selberg and Tesfahun \cite{ST}, who discovered 
that also in this case some of the nonlinearities have a null structure. Tesfahun \cite{T} improved this result to data without finite energy, namely for $(A(0),(\partial_t A)(0)) \in H^s \times H^{s-1}$ and $(F(0),(\partial_t F)(0)) \in H^l \times H^{l-1}$ with $s = \frac{6}{7}+\epsilon$ and $l = -\frac{1}{14}+\epsilon$ for any $\epsilon > 0$ by discovering an additional partial null structure. A further improvement was achieved by the author \cite{P2}, namely to $(s,l)=(\frac{5}{7}+\epsilon,-\frac{1}{7}+\epsilon)$ by modifying the solution spaces appropriately. In view of a recent result by S. Hong \cite{H1} who showed that the flow map is not $C^2$ if $s-1> 2r$ this result is in a sense sharp. This especially shows that the scaling critical regularity cannot be achieved by the used iteration method. S. Hong \cite{H} also proved local well-posedness of the Yang-Mills system in the Lorenz gauge for initial data in the Besov space $B^{\half}_{2,1} \times B^{-\half}_{2,1}$ , which is critical with respect to scaling, if an additional angular regularity is assumed.
Local well-posedness in energy space in 3D was also given by Oh \cite{O} using a new gauge, namely the Yang-Mills heat flow. He was also able to show that this solution can be globally extended \cite{O1}.  The Cauchy problem was also treated in higher space dimensions by several authors (\cite{KS},\cite{KT},\cite{KrT},\cite{KrSt},\cite{P1}).

 As the critical case in 3D  with respect to scaling is $(s,l)=(\half,-\half)$ , there is however still a gap, a phenomenon, which is also present in other gauges. 
In order to close this gap the author treated in \cite{P2} the local well-posedness problem for the Yang-Mills system in Lorenz gauge and space dimension $n = 3$ in the case of data $(A(0),(\partial_t A)(0)) \in \widehat{H}^{s,r} \times \widehat{H}^{s-1,r}$ and $(F(0),(\partial_t F)(0)) \in \widehat{H}^{l,r} \times \widehat{H}^{l-1,r}$ in Fourier-Lebesgue spaces for $r \neq 2$, which coincide with the classical Sobolev spaces $H^s$ for $r=2$. The assumption is that $s = \frac{16}{7r} - \frac{2}{7}+ \delta$ and $ l =\frac{15}{7r}-\frac{8}{7}+ \delta $ , where any  $\delta > 0$ is admissible.  Thus  $s \to 2+\delta$ and $l \to 1+\delta$ as $r \to 1$ , which is almost optimal with respect to scaling.

Such an approach was used by several authors already, starting with Vargas-Vega \cite{VV} for 1D Schr\"odinger equations. Gr\"unrock showed LWP for the modified KdV equation \cite{G}, a result which was improved by Gr\"unrock and Vega \cite{GV}. Gr\"unrock treated derivative nonlinear wave equations in 3+1 dimensions \cite{G1} and obtained an almost optimal result as $r \to 1$ with respect to scaling. Systems of nonlinear wave equations in the 2+1 dimensional case for nonlinearities which fulfill a null condition were considered by Grigoryan-Nahmod \cite{GN}. The latter two results are based on estimates by Foschi and Klainerman \cite{FK}.

The present paper is a continuation of \cite{P2}. Here we consider the local well-posedness problem for space dimension $n=2$. Our main result for the classical case of $L^2$-based data ($r=2$) is local well-posedness under the assumption $s > \frac{3}{4}$ and $l > -\frac{1}{4}$ , thus $\frac{3}{4}$ away from the critical exponents  $s=0$ and $l=-1$ with respect to scaling . In order to reduce this gap we again consider data in Fourier-Lebesgue spaces $\widehat{H}^{s,r} \times \widehat{H}^{l,r}$ . We obtain $ s \to \frac{3}{2}$ and $l \to \half $ as $r \to 1$ , which scales like $(s,l)= (\half,-\half)$ for the case $r=2$ . Thus the gap shrinks by $\frac{1}{4}$ .

The approach in the present paper is similar to \cite{P2}. In Chapter 2 we rewrite the Yang-Mills equations in Lorenz gauge as a system of semilinear wave equations. We also formulate the main theorems (Theorem \ref{Theorem1.1}, Cor. \ref{Cor} and Theorem \ref{Theorem1.3} and Cor. \ref{Cor3}). In chapter 3 we recall some basic facts about our solution spaces and a general local well-posedness theorem for the Cauchy problem for systems of nonlinear wave equations with data in Fourier-Lebesgue spaces, which allows to reduce it to estimates for the nonlinearities. In chapter 4 we give the final formulation of the system in terms of null forms as far as possible. The bi-, tri- and quadrilinear estimates sufficient for the local well-posedness result are formulated, where we rely on Tesfahun's paper \cite{T}. In chapter 5 we prove bilinear estimates for the null forms and for general bilinear terms in generalized Bourgain-Klainerman-Machedon spaces $H^r_{s,b}$ (and $X^r_{s,b,\pm}$) based on estimates by Foschi and Klainerman \cite{FK}, Gr\"unrock \cite{G}, Grigoryan-Nahmod \cite{GN} and Grigoryan-Tanguay \cite{GT}. In  chapter 6 we consider the case where $r>1$ is close to $1$ and prove the multilinear estimates formulated in chapter 4 by reduction to the bilinear estimates of chapter 5. In chapter 7 we prove these estimates in the classical case $r=2$ by reduction to bilinear estimates given by \cite{AFS}. Finally in chapter 8 we interpolate between the estimates for $r=2$ and $r=1+$ to obtain the desired local well-posedness result in the whole range $1 < r \le 2$ .

\section{Main results}
Expanding (\ref{1}) in terms of the gauge potentials $\left\{A_\alpha\right\}$, we obtain: 
 \begin{equation}\label{YM2}
  \square A_\beta = \partial_\beta\partial^\alpha A_\alpha- [\partial^\alpha A_\alpha, A_\beta] - [A^\alpha,\partial^\alpha A_\beta] - 
  [A^\alpha, F_{\alpha\beta}].
\end{equation}
If we now impose the Lorenz gauge condition,
the system \eqref{YM2} reduces to the nonlinear wave equation
 \begin{equation}\label{YM3}
  \square A_\beta = - [A^\alpha,\partial_\alpha A_\beta] -   [A^\alpha, F_{\alpha\beta}].
\end{equation}
In addition, regardless of the choice of gauge, $F$ satisfies the wave equation 
    \begin{equation}\label{YMF1}
 \begin{split}
   \square F_{\beta\gamma}&=-[A^\alpha,\partial_\alpha F_{\beta\gamma}] - \partial^\alpha[A_\alpha,F_{\beta\gamma}] - \left[A^\alpha,[A_\alpha,F_{\beta\gamma}]\right]
   \\
   & \quad - 2[F^{\alpha}_{{\;\;\,}\beta},F_{\gamma\alpha}] \, ,
   \end{split}
 \end{equation}
where we refer to \cite{ST}, chapter 3.2.

Expanding the second and fourth terms in \eqref{YMF1}, and also imposing the Lorenz gauge, yields 
\begin{equation}\label{YMF2}
\begin{split}
       \square F_{\beta\gamma}&= - 2[A^\alpha,\partial_\alpha F_{\beta\gamma}]
      + 2[\partial_\gamma A^\alpha, \partial_\alpha A_\beta]
      - 2[\partial_\beta A^\alpha, \partial_\alpha A_\gamma]
      \\
      &\quad + 2[\partial^\alpha A_\beta , \partial_\alpha A_\gamma]
                 + 2[\partial_\beta A^\alpha, \partial_\gamma A_\alpha] - [A^\alpha,[A_\alpha,F_{\beta\gamma}]] 
                 \\
                 &\quad  +2[F_{\alpha\beta},[A^\alpha,A_\gamma]]- 2[F_{\alpha\gamma},[A^\alpha,A_\beta]]
                      - 2[[A^\alpha,A_\beta],[A_\alpha,A_\gamma]] .  
                         \end{split}    
\end{equation}

 Note on the other hand by expanding the last term in the right hand side of \eqref{YM3}, we obtain 
 \begin{equation}\label{YM4}
  \square A_\beta = - 2[A^\alpha,\partial_\alpha A_\beta] + [A^\alpha,\partial_\beta A_\alpha] - 
  [A^\alpha, [A_\alpha,A_\beta]].
\end{equation}

We want to solve the system \eqref{YMF2}-\eqref{YM4} simultaneously for $A$ and $F$.
So to pose the Cauchy problem for this system, we consider initial data for $(A,F)$ at $t=0$:
\begin{equation}\label{Data-AF}
A(0) = a, \quad \partial_t A(0) = \dot a,
        \quad
    F(0) =f, \quad \partial_t F(0) = \dot f.
 \end{equation}

In fact, the initial data for $F$ can be determined from $(a, \dot a)$ as follows:
\begin{equation}\label{f}
\left\{
\begin{aligned}
  f_{ij} &= \partial_i a_j - \partial_j a_i + [a_i,a_j],
  \\
  f_{0i} &= \dot a_i - \partial_i a_0 + [a_0,a_i],
\\
  \dot f_{ij} &= \partial_i \dot a_j - \partial_j \dot a_i + [\dot a_i,a_j]+[ a_i, \dot a_j],
  \\
 \dot f_{0i} &= \partial^j f_{ji} +[a^\alpha, f_{\alpha i}] 
\end{aligned}
\right.
\end{equation}
where the first three expressions come from \eqref{curv} whereas 
the last one comes from (\ref{1}) with $\beta=i$.

Note that the Lorenz gauge condition $\partial^\alpha A_\alpha=0$ and (\ref{1}) with $\beta=0$ impose the constraints 
\begin{equation}\label{Const}
\dot a_0= \partial^i a_i,
\quad
   \partial^i f_{i0} + [a^i, f_{i0}] = 0 \, .
\end{equation}

Now we formulate our main theorem.
\begin{theorem}
\label{Theorem1.1}
Let $1 < r \le 2$ , $\epsilon > 0$. Assume that $s$ and $l$ satisfy the following conditions:
$$s = \frac{3}{2r} + \epsilon \, ,  \quad l =\frac{3}{2r}-1+ \epsilon \, . $$
Given initial data $(a,\dot{a}) \in \widehat{H}^{s,r} \times \widehat{H}^{s-1,r}$ , $(f,\dot{f}) \in \widehat{H}^{l,r} \times \widehat{H}^{l-1,r}$ ,  there exists a time $T > 0$ , $T=T(\|a\|_{\widehat{H}^{s,r}},\|\dot{a}\|_{\widehat{H}^{s-1,r}} , \|f\|_{\widehat{H}^{l,r}} , \|\dot{f}\|_{\widehat{H}^{l-1,r}})$, such that the Cauchy problem (\ref{YMF2}),(\ref{YM4}),(\ref{Data-AF}) has a unique solution $A_{\mu}\in X^r_{s,b,+}[0,T]+ X^r_{s,b,-}[0,T]$ , $F \in X^r_{l,b,+}[0,T]+ X^r_{l,b,-}[0,T]$ (these spaces are defined in Def. \ref{Def.}). Here $b = \frac{1}{r}+$ . This solution has the regularity
$$ A_{\mu} \in C^0([0,T],\widehat{H}^{s,r}) \cap C^1([0,T],\widehat{H}^{s-1,r}) \, , \, F \in C^0([0,T],\widehat{H}^{l,r}) \cap C^1([0,T],\widehat{H}^{l-1,r}) \, . $$
The solution depends continuously on the data and persistence of higher regularity holds.
\end{theorem}

\begin{Cor}
\label{Cor}
Let $s,r$ fulfill the assumptions of Theorem \ref{Theorem1.1}. Moreover assume that the initial data fulfill (\ref{f}) and (\ref{Const}). Given any $(a,\dot{a}) \in \widehat{H}^{s,r} \times \widehat{H}^{s-1,r}$ , there exists a time $T=T(\|a\|_{\widehat{H}^{s,r}},\|\dot{a}\|_{\widehat{H}^{s-1,r}} , \|f\|_{\widehat{H}^{l,r}} , \|\dot{f}\|_{\widehat{H}^{l-1,r}})$ , such that the solution $(A,F)$ of Theorem \ref{Theorem1.1} satisfies the Yang-Mills system (\ref{curv}),(\ref{1}) with Cauchy data $(a,\dot{a})$ and the Lorenz gauge condition $\partial^{\alpha} A_{\alpha} =0$ .
\end{Cor}
\begin{proof}[Proof of the Corollary]
 If $(a,\dot{a}) \in \widehat{H}^{s,r} \times \widehat{H}^{s-1,r}$ , then $(f,\dot{f})$ , defined by (\ref{f}), fulfill $(f,\dot{f}) \in \widehat{H}^{l,r} \times \widehat{H}^{l-1,r}$, as one easily checks. Thus we may apply Theorem \ref{Theorem1.1}.
The solution $(A,F)$ does not necessarily fulfill the Lorenz gauge condition and (\ref{curv}), i.e. $F=F[A]$ . If however the conditions (\ref{f}) and (\ref{Const}) are assumed then these properties are satisfied and $(A,F)$ is a solution of the Yang-Mills system (\ref{curv}),(\ref{1}) with Cauchy data $(a,\dot{a})$. This was shown in \cite{ST}, Remark 2. 
\end{proof}

Let us also formulate the result in the special case $r=2$ .

\begin{theorem}
	\label{Theorem1.3}
	Let  $\epsilon > 0$. Assume that $s$ and $l$ satisfy the following conditions:
	$s>\frac{3}{4}$ , $l>-\frac{1}{4}$ , $s\ge l\ge s-1$ , $2s-l > \frac{5}{4}$ , $4s-l > 3$ , $ 3s-2l > \frac{3}{2}$ and $ 2l-s > -\frac{5}{4}$ . 
	Given initial data $(a,\dot{a}) \in H^s \times H^{s-1}$ , $(f,\dot{f}) \in H^l \times H^{l-1} $ ,  there exists a time $T > 0$ , $T=T(\|a\|_{H^s},\|\dot{a}\|_{H^{s-1}} , \|f\|_{H^l}, \|\dot{f}\|_{H^{l-1}})$, such that the Cauchy problem (\ref{YMF2}),(\ref{YM4}),(\ref{Data-AF}) has a unique solution $A_{\mu}\in X^{s,b}_+[0,T]+ X^{s,b}_-[0,T]$ , $F \in X^{l,b}_+[0,T]+ X^{l,b}_-[0,T]$ (these spaces are defined in Def. \ref{Def.}). Here $b = \frac{1}{2}+$ . This solution has the regularity
	$$ A_{\mu} \in C^0([0,T],H^s) \cap C^1([0,T],H^{s-1}) \, , \, F \in C^0([0,T],H^l) \cap C^1([0,T],H^{l-1}) \, . $$
	The solution depends continuously on the data and persistence of higher regularity holds.
\end{theorem}
\begin{Cor}
	\label{Cor3}
	Let $s,r$ fulfill the assumptions of Theorem \ref{Theorem1.3}. Moreover assume that the initial data fulfill (\ref{f}) and (\ref{Const}). Given any $(a,\dot{a}) \in H^s \times H^{s-1}$ , there exists a time $T=T(\|a\|_{H^s},\|\dot{a}\|_{H^{s-1}} , \|f\|_{H^l} , \|\dot{f}\|_{H^{l-1}})$ , such that the solution $(A,F)$ of Theorem \ref{Theorem1.3} satisfies the Yang-Mills system (\ref{curv}),(\ref{1}) with Cauchy data $(a,\dot{a})$ and the Lorenz gauge condition $\partial^{\alpha} A_{\alpha} =0$ .
\end{Cor}

Let us fix some notation.
We denote the Fourier transform with respect to space and time  by $\,\,\widehat{}\,$ . 
 $\Box = \partial_t^2 - \Delta$ is the d'Alembert operator,
$a\pm := a \pm \epsilon$ for a sufficiently small $\epsilon >0$ , and $\langle \,\cdot\, \rangle := (1+|\cdot|^2)^{\frac{1}{2}}$ . \\
Let $\Lambda^{\alpha}$
be the multiplier with symbol  $
\langle\xi \rangle^\alpha $ . Similarly let $D^{\alpha}$,
and $D_{-}^{\alpha}$ be the multipliers with symbols $
\abs{\xi}^\alpha$ and $\quad ||\tau|-|\xi||^\alpha$ ,
respectively.

\begin{Def}
\label{Def.}
Let $1\le r\le 2$ , $s,b \in \R$ . The wave-Sobolev spaces $H^r_{s,b}$ are the completion of the Schwarz space ${\mathcal S}(\R^{1+3})$ with norm
$$ \|u\|_{H^r_{s,b}} = \| \langle \xi \rangle^s \langle  |\tau| - |\xi| \rangle^b \widehat{u}(\tau,\xi) \|_{L^{r'}_{\tau \xi}} \, , $$ 
where $r'$ is the dual exponent to $r$.
We also define $H^r_{s,b}[0,T]$ as the space of the restrictions of functions in $H^r_{s,b}$ to $[0,T] \times \mathbb{R}^3$.  Similarly we define $X^r_{s,b,\pm} $ with norm  $$ \|\phi\|_{X^r_{s,b\pm}} := \| \langle \xi \rangle^s \langle \tau \pm |\xi| \rangle^b \tilde{\phi}(\tau,\xi)\|_{L^{r'}_{\tau \xi}} $$ and $X^r_{s,b,\pm}[0,T] $ .\\
In the case $r=2$ we denote $H^2_{s,b} = : H^{s,b}$ and similarly $X^2_{s,b,\pm} = : X^{s,b}_{\pm}$ .
\end{Def}

\section{Preliminaries}

We start by collecting some fundamental properties of the solution spaces. We rely on \cite{G}. The spaces $X^r_{s,b,\pm} $ with norm  $$ \|\phi\|_{X^r_{s,b\pm}} := \| \langle \xi \rangle^s \langle \tau \pm |\xi| \rangle^b \tilde{\phi}(\tau,\xi)\|_{L^{r'}_{\tau \xi}} $$ for $1<r<\infty$ are Banach spaces with ${\mathcal S}$ as a dense subspace. The dual space is $X^{r'}_{-s,-b,\pm}$ , where $\frac{1}{r} + \frac{1}{r'} = 1$. The complex interpolation space is given by
$$(X^{r_0}_{s_0,b_0,\pm} , X^{r_1}_{s_1,b_1,\pm})_{[\theta]} = X^r_{s,b,\pm} \, , $$
where $s=(1-\theta)s_0+\theta s_1$, $\frac{1}{r} = \frac{1-\theta}{r_0} + \frac{\theta}{r_1}$ , $b=(1-\theta)b_0 + \theta b_1$ . Similar properties has the space $H^r_{s,b}$ .\\
If $u=u_++u_-$, where $u_{\pm} \in X^r_{s,b,\pm} [0,T]$ , then $u \in C^0([0,T],\hat{H}^{s,r})$ , if $b > \frac{1}{r}$ .

The "transfer principle" in the following proposition, which is well-known in the case $r=2$, also holds for general $1<r<\infty$ (cf. \cite{GN}, Prop. A.2 or \cite{G}, Lemma 1). We denote $ \|u\|_{\hat{L}^p_t(\hat{L}^q_x)} := \|\tilde{u}\|_{L^{p'}_{\tau} (L^{q'}_{\xi})}$ .
\begin{prop}
\label{Prop.0.1}
Let $1 \le p,q \le \infty$ .
Assume that $T$ is a bilinear operator which fulfills
$$ \|T(e^{\pm_1 itD} f_1, e^{\pm_2itD} f_2)\|_{\hat{L}^p_t(\hat{L}^q_x)} \lesssim \|f_1\|_{\hat{H}^{s_1,r}} \|f_2\|_{\hat{H}^{s_2,r}}$$
for all combinations of signs $\pm_1,\pm_2$ , then for $b > \frac{1}{r}$ the following estimate holds:
$$ \|T(u_1,u_2)\|_{\hat{L}^p_t(\hat{L}^q_x)} \lesssim \|u_1\|_{H^r_{s_1,b}}  \|u_2\|_{H^r_{s_2,b}} \, . $$
\end{prop}

The general local well-posedness theorem is the following (obvious generalization of)  \cite{G}, Thm. 1.
\begin{theorem}
\label{Theorem0.3}
Let $N_{\pm}(u,v):=N_{\pm}(u_+,u_-,v_+,v_-)$ and $M_{\pm}(u,v):=M_{\pm}(u_+,u_-,\\v_+,v_-)$ be multilinear functions.
Assume that for given $s,l \in \R$, $1 < r < \infty$ there exist $ b,a > \frac{1}{r}$ such that the estimates
$$ \|N_{\pm}(u,v)\|_{X^r_{s,b-1+,\pm}} \le \omega_1( \|u\|_{X^r_{s,b}},\|v\|_{X^r_{l,a}}) $$
and 
$$\|M_{\pm}(u,v)\|_{X^r_{l,a-1+,\pm}} \le \omega_2( \|u\|_{X^r_{s,b}},\|v\|_{X^r_{l,a}}) $$
are valid with nondecreasing functions $\omega_j$ , where $\|u\|_{X^r_{s,b}} := \|u_-\|_{X^r_{s,b,-}} + \|u_+\|_{X^r_{s,b,+}}$. Then there exist $T=T(\|u_{0_ {\pm}}\|_{\hat{H}^{s,r}},\|v_{0_{\pm}}\|_{\hat{H}^{l,r}})$ $>0$ and a unique solution $(u_+,u_-,\\v_+,v_-) \in X^r_{s,b,+}[0,T] \times X^r_{s,b,-}[0,T] \times X^r_{l,a,+}[0,T] \times X^r_{l,a,-}[0,T] $ of the Cauchy problem
$$ \partial_t u_{\pm} \pm i\Lambda u = N_{\pm}(u,v) \quad , \quad \partial_t v_{\pm} \pm i\Lambda v = M_{\pm}(u,v) $$ $$         u_{\pm}(0) = u_{0_{\pm}} \in \hat{H}^{s,r} \quad , \quad v_{\pm}(0) = v_{0_{\pm}} \in \hat{H}^{l,r}       \, . $$
 This solution is persistent and the mapping data upon solution $(u_{0+},u_{0-},v_{0+},v_{0-}) \\ \mapsto (u_+,u_-,v_+,v_-)$ , $\hat{H}^{s,r} \times \hat{H}^{s,r}\times \hat{H}^{l,r} \times \hat{H}^{l,r} \to X^r_{s,b,+}[0,T_0] \times X^r_{s,b,-}[0,T_0] \times X^r_{l,a,+}[0,T_0]\times X^r_{l,a,-}[0,T_0] $ is locally Lipschitz continuous for any $T_0 < T$.
\end{theorem}

\section{Reformulation of the problem and null structure}
The reformulation of the Yang-Mills equations and the reduction of our main theorem to nonlinear estimates is completely taken over from Tesfahun \cite{T} (cf. also the fundamental paper by Selberg and Tesfahun \cite{ST}).

The standard null forms are given by
\begin{equation}\label{OrdNullforms}
\left\{
\begin{aligned}
Q_{0}(u,v)&=\partial_\alpha u \partial^\alpha v=-\partial_t u \partial_t v+\partial_i u \partial^j v,
\\
Q_{\alpha\beta}(u,v)&=\partial_\alpha u \partial_\beta v-\partial_\beta u \partial_\alpha v.
\end{aligned}
        \right.
\end{equation}
For $ g$-valued $u,v$, define a commutator version of null forms by 
\begin{equation}\label{CommutatorNullforms}
\left\{
\begin{aligned}
  Q_0[u,v] &= [\partial_\alpha u, \partial^\alpha v] = Q_0(u,v) - Q_0(v,u),
  \\
  Q_{\alpha\beta}[u,v] &= [\partial_\alpha u, \partial_\beta v] - [\partial_\beta u, \partial_\alpha v] = Q_{\alpha\beta}(u,v) + Q_{\alpha\beta}(v,u).
\end{aligned}
\right.
\end{equation}

 Note the identity
\begin{equation}\label{NullformTrick}
  [\partial_\alpha u, \partial_\beta u]
  = \frac12 \left( [\partial_\alpha u, \partial_\beta u] - [\partial_\beta u, \partial_\alpha u] \right)
  = \frac12 Q_{\alpha\beta}[u,u].
\end{equation}

Define 
\begin{equation}\label{NewNull} 
  \mathcal{Q}[u,v] = Q_{12}[R_1 u_2-R_2 u_1,\phi] - Q_{0i}[R^i u_0,v] \, ,
\end{equation}
where 
$R_i = \Lambda^{-1}\partial_i $ are the Riesz transforms.

We split the spatial part $\mathbf A=(A_1,A_2, A_3)$ of the potential into divergence-free and curl-free parts and a smoother part:
\begin{equation}\label{SplitA}  
\mathbf A = \mathbf A^{\text{df}} + \mathbf A^{\text{cf}} + \Lambda^{-2} \mathbf A,
\end{equation}
where
\begin{align*}
\mathbf A^{\text{df}}&=  (\partial_2(\partial_1 A_2-\partial_2 A_1),-\partial_1(\partial_1 A_2-\partial_2 A_1))
\\
\mathbf A^{\text{cf}}&= -\Lambda^{-2} \nabla (\nabla \cdot \mathbf A).
\end{align*}

\subsection{ Terms of the form $[A^\alpha,\partial_\alpha \phi]$ and $ [\partial_tA^\alpha, \partial_\alpha\phi]$ }

In the Lorenz gauge, terms of the form  $[A^\alpha,\partial_\alpha \phi] $, where $A_\alpha,\phi \in \mathcal S$
with values in $\mathfrak g$, can be shown to be a sum of bilinear null forms 
and a smoother bilinear part whereas the term $ [\partial_tA^\alpha, \partial_\alpha\phi]$ is a null form.  
\begin{lemma}\label{Lemma-Null0} In the Lorenz gauge, we have 
	the identities
	\begin{align}
	\label{Null0}
	[A^\alpha, \partial_\alpha \phi ] & =
	\mathfrak Q\left[\Lambda^{-1} A,\phi\right] + [\Lambda^{-2}  A^\alpha, \partial_\alpha \phi ],
	\\
	\label{Null1}
	[\partial_tA^\alpha, \partial_\alpha\phi]&=  Q_{0i}\left[A^i , \phi \right].
	\end{align}
	
\end{lemma}
\begin{proof}
	To show \eqref{Null0} we modify the proof in \cite{ST},
Lemma 1	whereas \eqref{Null1} is proved in the same paper
	(see identity (2.7) therein). 
	
	Using \eqref{SplitA} we write
	\begin{align*}
	A^\alpha \partial_\alpha \phi
	&=
	\left( - A_0 \partial_t \phi
	+ \mathbf A^{\text{cf}} \cdot \nabla \phi \right)
	+ \mathbf A^{\text{df}} \cdot \nabla \phi +  \Lambda^{-2} \mathbf A  \cdot \nabla \phi
	\end{align*}
	Let us first consider the first term in the parentheses. 
	We use the Lorenz gauge, $\partial_t A_0=\nabla \cdot \mathbf A  $,  to write
	\begin{align*}
	\mathbf A^{\text{cf}} \cdot \nabla \phi&
	=-\Lambda^{-2} \partial^i(\partial_t A_0) \partial_i \phi=
	- \partial_t ( \Lambda^{-1} R^i A_0)\partial_i \phi.
	\end{align*}
	We can also write
	\begin{align*}
	A_0 \partial_t \phi 
	&=-\Lambda^{-2}  \partial_i\partial^i A_0 \partial_t \phi+ \Lambda^{-2}  A_0 \partial_t \phi
	\\
	&=-\partial_i(\Lambda^{-1} R^i A_0) \partial_t \phi
	+  \Lambda^{-2}  A_0 \partial_t \phi.
	\end{align*}
	Combining the above identities, we get
	\begin{align*}
	-A_0 \partial_t \phi +  \mathbf A^{\text{cf}} \cdot \nabla \phi
	&=Q_{i0}(\Lambda^{-1} R^i A_0, \phi)- \Lambda^{-2}  A_0 \partial_t \phi.
	\end{align*}
	
	Next, we consider the second term. 
	We have 
	\begin{align*}
	\mathbf A^{\text{df}} \cdot \nabla \phi
	&= \Lambda^{-2}(\partial_2(\partial_1A_2-\partial_2 A_1) \partial_1 \phi - \partial_1(\partial_1 A_2 - \partial_2 A_1) \partial_2 \phi)
	\\
	&=- \Lambda^{-2} Q_{12}(\partial_1 A_2-\partial_2 A_1,\phi)
	\\
	&=
	-  Q_{12}\left(\Lambda^{-1}( R^1 A_2 -R_2 A_1), \phi\right).
	\end{align*}
	
	Thus, we have shown 
	\begin{equation}\label{Null2}
	\begin{split}
	A^\alpha \partial_\alpha \phi
	&= -  Q_{12}\left(\Lambda^{-1}( R^1 A_2 -R_2 A_1), \phi\right)
	 \\
	& \qquad + Q_{i0}(\Lambda^{-1} R^i A_0, \phi)+\Lambda^{-2}  A^\alpha \partial_\alpha \phi.
	\end{split}
	\end{equation} 
	Similarly, modifying the above argument one can show 
	\begin{equation}\label{Null3}
	\begin{split}
	\partial_\alpha \phi A^\alpha
	&= -  Q_{12}\left(\phi,\Lambda^{-1}( R^1 A_2 -R_2 A_1)\right)
	\\
	& \qquad + Q_{i0}(\phi,\Lambda^{-1} R^i A_0)+ \partial_\alpha \phi \Lambda^{-2}  A^\alpha .
	\end{split}
	\end{equation}   
	Subtracting \eqref{Null2} and \eqref{Null3} yields \eqref{Null0}.
\end{proof}

\subsection{ Terms of the form
	$[A^\alpha,\partial_\beta A_\alpha]$ }
 In the Lorenz gauge, this term can be 
written as a sum of bilinear null form terms, bilinear terms which are smoother,  
a bilinear term which contains only $F$ and higher order terms in $(A, F)$.

\begin{lemma}\label{Lemma-null1} In the Lorenz gauge, we have 
	the identity
	\begin{align*}\label{Null4}
	[A^\alpha,\partial_\beta A_\alpha]&
	=\sum_{i=1}^4\Gamma^i_\beta(A, \partial A, F, \partial F),
	\end{align*}
\end{lemma}
where
\begin{equation}\label{Gammas}
\left\{
\begin{aligned}
\Gamma^1_\beta(A, \partial A,  F, \partial F)
&=-[A_0,\partial_\beta A_0] + 
[\Lambda^{-1} R_j (\partial_t A_0),  \Lambda^{-1} R^j  \partial_t (\partial_\beta A_0)] ,
\\
\Gamma^2_\beta(A, \partial A,  F, \partial F)
&= - \Big\{ Q_{12}[\Lambda^{-1}  R^n \mathbf A_n ,\Lambda^{-1} (R^1 \partial_\beta A^2 -R_2 \partial_{\beta} A_1) ] \\
&\hspace{6em}+ 
Q_{12}[\Lambda^{-1}  R^n \partial_\beta  A_n , \Lambda^{-1}  (R^1 A_2 - R_2 A_1) ]
\Big\},
\\
\Gamma^3_\beta(A, \partial A,  F, \partial F)&=[ \Lambda^{-2} \partial_j  F_{12} ,
\Lambda^{-2} \partial_{\beta} \partial^j F_{12}]
\\
&\quad
-[\Lambda^{-2}\partial_j  F_{12} , 
\Lambda^{-2} \partial_\beta \partial^j [ A_1,  A_2]]
\\
& \quad- [ \Lambda^{-2} \partial_j [ A_1, A_2] ,
\Lambda^{-2} \partial_\beta \partial^j F_{12} ] 
\\
&\quad+ [\Lambda^{-2}\partial_j [A_1,A_2],
\Lambda^{-2}\partial_\beta \partial^j  [A_1,A_2] ],
\\
\Gamma^4_\beta(A, \partial A,  F, \partial F)&=[\mathbf A^{\text{cf}} + \mathbf A^{\text{df}}, \Lambda^{-2} \partial_\beta \mathbf A] 
+ [\Lambda^{-2}\mathbf A,  \partial_\beta \mathbf A  ].
\end{aligned}
\right.
\end{equation}

Thus, $\Gamma^2_\beta$ is a combination of the commutator version $Q$-type null forms.
The term $\Gamma^1_\beta$ is also a null form (of non $Q$-type) as shown below.

\subsection{The system \eqref{YMF2}--\eqref{YM4} in terms of the null forms}
In view of Lemma \ref{Lemma-Null0} the first, second and third bilinear terms in \eqref{YMF2} 
are null forms up to some smoother bilinear terms. 
By the identity \eqref{NullformTrick}, the fourth and fifth
terms are identical to $2Q_0[A_\beta,A_\gamma]$ and $Q_{\beta\gamma}[A^\alpha,A_\alpha]$, respectively.

By Lemma \ref{Lemma-Null0}, the first term in \eqref{YM4} is a null form up to some smoother bilinear terms. 
By Lemma \ref{Lemma-null1} the second term in \eqref{YM4} is a sum of bilinear null form terms, bilinear terms 
which are smoother,  
a bilinear term which contains only $F$ and higher order terms in $(A, F)$. 

Thus the system (\ref{YMF2}),(\ref{YM4}) in Lorenz gauge can be written in the following form
\begin{equation}\label{AF}
\begin{aligned}
  \square A_\beta &=  \mathcal M_\beta(A,\partial_t A,F,\partial_t F),
  \\
  \square F_{\beta\gamma} &=  \mathcal N_{\beta\gamma}(A,\partial_t A,F,\partial_t F),
\end{aligned}
\end{equation}
where
\begin{align*}
  \mathcal M_\beta(A,\partial_t A,F,\partial_t F) &= -2 \mathcal Q[\Lambda^{-1} A,A_\beta] +
  \sum_{i=1}^4\Gamma^i_\beta(A, \partial A, F, \partial F)-2[\Lambda^{-2}  A^\alpha, \partial_\alpha A_\beta ]
  \\
 &\quad  - [A^\alpha, [A_\alpha, A_\beta]],
   \end{align*}

  \begin{align*}
  \mathcal N_{12}(A,\partial_t A,F,\partial_t F)
  = &- 2\mathcal Q[\Lambda^{-1} A,F_{12}]
  + 2\mathcal Q[\Lambda^{-1} \partial_2 A, A_1]- 2\mathcal Q[\Lambda^{-1} \partial_1 A, A_2] 
  \\
  & + 2Q_0[A_1 , A_2]
  + Q_{12}[A^\alpha,A_\alpha]-2[\Lambda^{-2}  A^\alpha, \partial_\alpha F_{12} ]
  \\
  &+2[\Lambda^{-2}  \partial_2A^\alpha, \partial_\alpha A_{1} ]-2[\Lambda^{-2}  \partial_1 A^\alpha, \partial_\alpha A_{2} ]
  \\
  & - [A^\alpha,[A_\alpha,F_{12}]] + 2[F_{\alpha 1},[A^\alpha,A_2]] - 2[F_{\alpha 2},[A^\alpha,A_1]]
  \\
  & - 2[[A^\alpha,A_1],[A_\alpha,A_2]],
  \end{align*} 
  
  \begin{align*}
  \mathcal N_{0i}(A,\partial_t A,F,\partial_t F)
  = &- 2\mathcal Q[\Lambda^{-1} A,F_{0i}]
  + 2\mathcal Q[\Lambda^{-1} \partial_i A, A_0]- 2 Q_{0j}[A^j,A_i] \\
	&+ 2Q_0[A_0 , A_i]
  + Q_{0i}[A^\alpha,A_\alpha]-2[\Lambda^{-2}  A^\alpha, \partial_\alpha F_{0i} ]\\
	&+2[\Lambda^{-2}  \partial_i A^\alpha, \partial_\alpha A_{0} ]
   - [A^\alpha,[A_\alpha,F_{0i}]] + 2[F_{\alpha 0},[A^\alpha,A_i]] \\ 
	&- 2[F_{\alpha i},[A^\alpha,A_0]]
   - 2[[A^\alpha,A_0],[A_\alpha,A_i]]
  \end{align*}

In a standard way we rewrite the system (\ref{AF}) as a first order (in t) system. Defining
$A_{\pm} = \half(A \pm (i\Lambda)^{-1} \partial_t A) \quad , \quad F_{\pm} = \half(F \pm (i \Lambda)^{-1}\partial_t F) $ , so that $A=A_++A_-$ , $\partial_t A= i \Lambda(A_+-A_-)$ , $F=F_++F_-$ , $\partial_t F = i \Lambda (F_+-F_-)$ the system transforms to
\begin{align}
(i \partial_t \pm \Lambda)A^{\beta}_{\pm} & = -A^{\beta} \mp(2 \Lambda)^{-1} \mathcal M_\beta(A,\partial_t A,F,\partial_t F)\, ,\\
(i \partial_t \pm \Lambda)F^{\beta \gamma}_{\pm} & = -F^{\beta \gamma} \mp(2 \Lambda)^{-1} \mathcal N_{\beta\gamma}(A,\partial_t A,F,\partial_t F) \, .
\end{align}
The initial data transform to 
$$A_{\pm}(0)= \half(a \pm (i\Lambda)^{-1} \dot{a}) \in 
\widehat{H}^{s,r} \quad , \quad F_{\mp}(0)= \half(f \pm (i\Lambda)^{-1} \dot{f}) \in \widehat{H}^{l,r} \, .$$

Now, looking at the terms in $ \mathcal{M}_{\beta}$ and $ \mathcal{N}_{\beta \gamma}$ and noting the fact that the Riesz transforms 
$R_i$ are bounded in the spaces involved, the estimates in Theorem \ref{Theorem0.3} 
reduce to proving:\\
1. the estimates for the null forms $Q_{12}$ , $Q_0$ and $Q \in \{Q_{0i},Q_{12}\}$ :
\begin{align}
  \label{21}
  \norm{ Q[\Lambda^{-1} A, A]}_{H^r_{s-1,b-1+}}
  &\lesssim \|A\|_{X^r_{s,b}} \|A\|_{X^r_{s,b}},
  \\
    \label{22}
  \norm{  Q_{12}[\Lambda^{-1} A, \Lambda^{-1} \partial A]}_{H^r_{s-1,b-1+}}
  &\lesssim  \|A\|_{X^r_{s,b}} \|A\|_{X^r_{s,b}} ,
  \\
 \label{23}
  \norm{ Q[\Lambda^{-1}A, F]}_{H^r_{l-1,a-1+} }
  &\lesssim \|A\|_{X^r_{s,b}}  \|F\|_{X^r_{l,a}},
    \\
  \label{24}
  \norm{  Q[ A,   A]}_{H^r_{l-1,a-1+} }
  &\lesssim \|A\|_{X^r_{s,b}} \|A\|_{X^r_{s,b}} ,\\
	\label{25}
  \norm{ Q_0[ A,   A]}_{H^r_{l-1,a-1+} }
  &\lesssim \|A\|_{X^r_{s,b}}  \|A\|_{X^r_{s,b}} ,
	\end{align} 
the following estimate for $\Gamma^1$ and other bilinear terms
\begin{align}
  \label{26}
  \norm{\Gamma^1( A, \partial A)}_{H^r_{s-1,b-1+}}
  &\lesssim\|A\|_{X^r_{s,b}} \|A\|_{X^r_{s,b}} ,
  \\
  \label{27} 
   \norm{\Pi( A, \Lambda^{-2} \partial A  ) }_{H^r_{s-1,b-1+}}
     &\lesssim \|A\|_{X^r_{s,b}} \|A\|_{X^r_{s,b}} ,
     \\
     \label{28} 
   \norm{ \Pi( \Lambda^{-2} A,  \partial A)   }_{H^r_{s-1,b-1-+}}
     &\lesssim \|A\|_{X^r_{s,b}} \|A\|_{X^r_{s,b}} ,
     \\
  \label{29} 
  \norm{\Pi(\Lambda^{-1} F, \Lambda^{-1} \partial F  ) }_{H^r_{s-1,b-1+}}
     &\lesssim \|F\|_{X^r_{l,a}} \|F\|_{X^r_{l,a}},
     \\
       \label{30} 
   \norm{\Pi( \Lambda^{-2} A,  \partial F)   }_{H^r_{l-1,a-1+}}
     &\lesssim \|A\|_{X^r_{s,b}} \|F\|_{X^r_{l,a}},
     \\
       \label{31} 
   \norm{ \Pi( \Lambda^{-1} A,  \partial A)   }_{H^r_{l-1,a-1+}}
     &\lesssim \|A\|_{X^r_{s,b}} \|A\|_{X^r_{s,b}}
     \end{align}
and\\
2. the following trilinear and quadrilinear estimates:
 \begin{align}
   \label{32}
   \norm{\Pi(\Lambda^{-1} F,\Lambda^{-1} \partial( AA) )}_{H^r_{s-1,b-1+}}
  &\lesssim  \|F\|_{X^r_{l,a}} \|A\|_{X^r_{s,b}} \|A\|_{X^r_{s,b}} ,
  \\
   \label{33}
   \norm{\Pi(\Lambda^{-1}\partial F, \Lambda^{-1}  ( AA) )}_{H^r_{s-1,b-1+}}
  &\lesssim  \|F\|_{X^r_{l,a}} \|A\|_{X^r_{s,b}} \|A\|_{X^r_{s,b}} ,
  \\
  \label{34}
   \norm{\Pi(\Lambda^{-1}(AA), \Lambda^{-1} \partial ( AA)) }_{H^r_{s-1,b-1+}}
  &\lesssim  \|A\|_{X^r_{s,b}} \|A\|_{X^r_{s,b}} \|A\|_{X^r_{s,b}} \|A\|_{X^r_{s,b}} ,
  \\
   \label{35} 
  \norm{\Pi(A,A,A)}_{H^r_{s-1,b-1+}}
  &\lesssim\|A\|_{X^r_{s,b}} \|A\|_{X^r_{s,b}} \|A\|_{X^r_{s,b}},
  \\
   \label{36}
  \norm{\Pi(A, A, F)}_{H^r_{l-1,a-1+}}
  &\lesssim \|A\|_{X^r_{s,b}} 
  \|A\|_{X^r_{s,b}}\|F\|_{X^r_{l,a}},
  \\
   \label{37}
  \norm{\Pi(A,A, A, A)}_{H^r_{l-1,a-1+}}
  &\lesssim  \|A\|_{X^r_{s,b}} \|A\|_{X^r_{s,b}} \|A\|_{X^r_{s,b}} \|A\|_{X^r_{s,b}} \, .
     \end{align}
$\Pi(\cdots)$ denotes a multilinear operator in its arguments and
$\|u\|_{X^r_{s,b}} := \|u_-\|_{X^r_{s,b,-}} + \|u_+\|_{X^r_{s,b,+}}$ .

The matrix commutator null forms are linear combinations of the ordinary ones, 
in view of \eqref{CommutatorNullforms}. Since the matrix 
structure plays no role in the estimates under consideration, 
we reduce (\ref{21})--(\ref{25}) to estimates of the ordinary null forms for $\mathbb C$-valued 
functions $u$ and $v$ (as in \eqref{OrdNullforms}).

Next we consider the term $\Gamma_{\beta}^1$ and want to show that it is a null form. In fact the detection of this null structure was the main progress of his paper over Selberg-Tesfahun \cite{ST}.  

We may ignore its matrix form and treat 
$$\Gamma^1_k(A_0,, \partial_k A_0)
=-A_0 (\partial_k A_0) + 
\Lambda^{-1} R_j (\partial_t A_0)  \Lambda^{-1} R^j  \partial_t (\partial_k A_0))$$
for $k=1,2,3$ and
\begin{align*}
\Gamma^1_0(A_0,, \partial^i A_i)
&=-A_0 (\partial_0 A_0) + 
\Lambda^{-1} R_j (\partial_t A_0)  \Lambda^{-1} R^j  \partial_t (\partial_0 A_0)) \\
& = -A_0 (\partial^i A_i) + 
\Lambda^{-1} R_j (\partial_t A_0)  \Lambda^{-1} R^j  \partial_t (\partial^i A_i)) \,,
\end{align*}
where we used the Lorenz gauge $\partial_0 A_0 = \partial^i A_i$ in the last line in order to eliminate one time derivative. Thus we have to consider
$$ \Gamma^1(u,v) = -uv + \Lambda^{-1} R_j (\partial_t u) \Lambda^{-1} R^j(\partial_t v) \, , $$
where $u=A_0$ and $v=\partial^i A_i$ or $v=\partial_k A_0$ .

The proof of the following lemma was essentially given by Tesfahun \cite{T}. 
\begin{lemma}
\label{Lemma2.1}
Let $q_{12}(u,v) := Q_{12}(D^{-1}u,D^{-1}v)$ , $q_0(u,v) := Q_0(D^{-1}u,D^{-1}v)$ .
The following estimate holds:
\begin{align}
\label{45'}
\Gamma^1(u,v) & q_{12}(u,v) + q_0(u,v) + (\Lambda^{-2}u)v + u(\Lambda^{-2}v) \, .  
\end{align}
Here $u \preceq v$ means $\abs{\widehat u} \lesssim \abs{\widehat v}$ .
\end{lemma}
\begin{proof}
$\Gamma^1(u,v)$ has the symbol
\begin{align*}
p(\xi,\tau,\eta,\lambda)& = -1 + \frac{\angles{\xi,\eta} \tau \lambda}{\angles{\xi}^2 \angles{\eta}^2} = \left( -1 + \frac{\angles{\xi,\eta} \angles{\xi,\eta}}{\angles{\xi}^2 \angles{\eta}^2}\right) + \frac{(\tau \lambda -\angles{\xi,\eta}) \angles{\xi,\eta}}{\angles{\xi}^2 \angles{\eta}^2}  = I + II
\end{align*}
Now we estimate
\begin{align*}
|I| & = \left| \frac{|\xi|^2 |\eta|^2\cos^2 \angle(\xi,\eta)}{\angles{\xi}^2 \angles{\eta}^2} -1 \right| \\
&
\le \left| \frac{\angles{\xi}^2 \angles{\eta}^2 \cos^2 \angle(\xi,\eta)}{\angles{\xi}^2 \angles{\eta}^2} -1 \right|
+ \left| \frac{|\xi|^2 |\eta|^2 - \angles{\xi}^2 \angles{\eta}^2}{\angles{\xi}^2 \angles{\eta}^2} \right| \\
& = \sin^2 \angle(\xi,\eta) + \left| \frac{|\xi|^2 |\eta|^2 - \angles{\xi}^2 \angles{\eta}^2}{\angles{\xi}^2 \angles{\eta}^2} \right| \, ,
\end{align*}
where $\angle(\xi,\eta)$ denotes the angle between $\xi$ and $\eta$ .
We have
$$\left| \frac{|\xi|^2 |\eta|^2 - \angles{\xi}^2 \angles{\eta}^2}{\angles{\xi}^2 \angles{\eta}^2} \right| = \frac{|\xi|^2 + |\eta|^2 + 1}{\angles{\xi}^2 \angles{\eta}^2} \le \frac{1}{\angles{\xi}^2} + \frac{1}{\angles{\eta}^2} $$
and
$$|\sin \angle(\xi,\eta)| = \frac{|\xi_1 \eta_2 - \xi_2 \eta_1|}{|\xi|\,|\eta|} \, . $$
 Thus the operator belonging to the symbol I is controlled by 
$ q_{12}(u,v) +  (\Lambda^{-2}u)v + u(\Lambda^{-2}v)$ .
Moreover 
$$ |II| \le \frac{|\tau \lambda - \angles{\xi,\eta}|}{\angles{\xi} \angles{\eta}} \le |q_0(\xi,\eta)| \, .$$ 
Thus we obtain (\ref{45'}).
\end{proof}

\section{Bilinear estimates}
The proof of the following bilinear estimates relies on estimates given by Foschi and Klainerman \cite{FK}. We first treat the case $r>1$ , but close to $1$.
\begin{lemma}
\label{Lemma5.1}
Assume $0 \le\alpha_1,\alpha_2 $ ,  $\alpha_1+\alpha_2 \ge \frac{1}{r}$ and $ b > \frac{1}{r}$. The following estimate applies
$$ \|q_{12}(u,v)\|_{H^r_{0,0}} \lesssim \|u\|_{X^r_{\alpha_1,b,\pm_1}} \|v\|_{X^r_{\alpha_2,b,\pm_2}} \, . $$
\end{lemma}
\begin{proof}
	Because we use inhomogeneous norms it is obviously possible to assume $\alpha_1 + \alpha_2 = \frac{1}{r}$ . Moreover, by interpolation we may reduce to the case $\alpha_1= \frac{1}{r}$ , $\alpha_2 =0$ .
	
The left hand side of the claimed estimate equals
$$ \|{\mathcal F}(q_{12}(u,v))\|_{L^{r'}_{\tau \xi}} = \| \int q_{12}(\eta,\eta-\xi) \tilde{u}(\lambda,\eta) \tilde{v}(\tau - \lambda,\xi - \eta) d\lambda d\eta \|_{L^{r'}_{\tau \xi}} \, . $$
Let now $u(t,x) = e^{\pm_1 iD} u_0^{\pm_1}(x)$ , $v(t,x) = e^{\pm_2 iD} v_0^{\pm_2}(x)$ , so that 
$$ \tilde{u}(\tau,\xi) = c \delta(\tau \mp_1 |\xi|) \widehat{u_0^{\pm_1}}(\xi) \quad , \quad \tilde{v}(\tau,\xi) = c \delta(\tau \mp_2 |\xi|) \widehat{v_0^{\pm_2}}(\xi) \, . $$
This implies
\begin{align*}
&\|{\mathcal F}(q_{12}(u,v))\|_{L^{r'}_{\tau \xi}} \\
&= c^2 \| \int q_{12}(\eta,\eta-\xi) \widehat{u_0^{\pm_1}}(\eta) \widehat{v_0^{\pm_2}}(\xi-\eta) \,\delta(\lambda \mp_1 |\eta|) \delta(\tau-\lambda\mp_2|\xi-\eta|) d\lambda d\eta \|_{L^{r'}_{\tau \xi}} \\
& = c^2 \| \int q_{12}(\eta,\eta-\xi) \widehat{u_0^{\pm_1}}(\eta) \widehat{v_0^{\pm_2}}(\xi-\eta) \,\delta(\tau\mp_1|\eta| \mp_2|\xi-\eta|) d\eta \|_{L^{r'}_{\tau \xi}} \, .
\end{align*}
By symmetry we only have to consider the  elliptic case $\pm_1=\pm_2 = +$ and the hyperbolic case $\pm_1= + \, , \, \pm_2=-$ .  \\
{\bf Elliptic case.} We obtain by \cite{FK}, Lemma 13.2:
$$|q_{12}(\eta,\xi-\eta)| \le \frac{|\eta_1 (\xi - \eta)_2 - \eta_2 (\xi-\eta)_1|}{|\eta| \, |\xi - \eta|} \lesssim \frac{|\xi|^{\half} (|\eta| + |\xi - \eta| - |\xi|)^{\half}}{|\eta|^{\half} |\xi - \eta|^{\half}} \, . $$
By H\"older's inequality we obtain
\begin{align*}
&\|{\mathcal F}(q_{12}(u,v))\|_{L^{r'}_{\tau \xi}} \\
& \lesssim \|\int \frac{|\xi|^{\half} ||\tau|-|\xi||^{\half}}{|\eta|^{\half} |\xi - \eta|^{\half}} \,
 \delta(\tau-|\eta|-|\xi - \eta|) \, |\widehat{u_0^+}(\eta)| \, |\widehat{v_0^+}(\xi - \eta)| d\eta \|_{L^{r'}_{\tau \xi}} \\
& \lesssim \sup_{\tau,\xi} I  \,\, \|\widehat{D^{\frac{1}{r}} u_0^+}\|_{L^{r'}} \| \widehat{v_0^+}\|_{L^{r'}} \, ,
\end{align*}
where
$$ I = |\xi|^{\half} ||\tau|-|\xi||^{\half} \left( \int \delta(\tau - |\eta| - |\xi - \eta|) \, |\eta|^{-1-\frac{r}{2}} |\xi - \eta|^{-\frac{r}{2}} d\eta \right)^{\frac{1}{r}} \, . $$
We want to prove $ \sup_{\tau,\xi} I \lesssim 1 $ . By \cite{FK}, Lemma 4.3 we obtain
$$\int \delta(\tau - |\eta| - |\xi - \eta|) \, |\eta|^{-1-\frac{r}{2}} |\xi - \eta|^{-\frac{r}{2}} d\eta \sim \tau^A ||\tau|-|\xi||^B \, , $$
where $A= \max(1+\frac{r}{2},\frac{3}{2}) - 1-r= -\frac{r}{2}$ and $B=1-\max(1+\frac{r}{2},\frac{3}{2})=-\frac{r}{2}$ . Using $|\xi| \le |\tau|$  this implies
$$
I  \lesssim |\xi|^{\half} ||\tau|-|\xi||^{\half} \tau^{-\half} ||\tau|-|\xi||^{-\half}\le  1 \, .$$
{\bf Hyperbolic case.} We start with the following bound (cf. \cite{FK}, Lemma 13.2):
$$  |q_{12}(\eta,\xi-\eta)| \le \frac{|\eta_1 (\xi - \eta)_2 - \eta_2 (\xi-\eta)_1|}{|\eta| \, |\xi - \eta|}  \lesssim \frac{|\xi|^{\half} (|\xi|-||\eta|-|\eta-\xi||)^{\half}}{|\eta|^{\half} |\xi-\eta|^{\half}} \, , $$
so that similarly as in the elliptic case we have to estimate
$$ I = |\xi|^{\half} ||\tau|-|\xi||^{\half} \left( \int \delta(\tau - |\eta| + |\xi - \eta|) \, |\eta|^{-1-\frac{r}{2}} |\xi - \eta|^{-\frac{r}{2}} d\eta \right)^{\frac{1}{r}} \, . $$
In the subcase $|\eta|+|\xi-\eta| \le 2|\xi|$ we apply \cite{FK}, Prop. 4.5 and obtain
$$\int \delta(\tau - |\eta| + |\xi - \eta|) \, |\eta|^{-1-\frac{r}{2}} |\xi - \eta|^{-\frac{r}{2}} d\eta \sim |\xi|^A ||\xi|-|\tau||^B \, . $$
where in the subcase $0 \le \tau \le |\xi|$ we obtain $A=\max(\frac{r}{2},\frac{3}{2}) - 1-r = \half -r$ and $B= 1- \max(\frac{r}{2},\frac{3}{2})= -\frac{1}{2}$. \\
This implies
$$I \lesssim |\xi|^{\half} ||\tau|-|\xi||^{\half} |\xi|^{\frac{1}{2r}-1} ||\tau|-|\xi||^{-\frac{1}{2r}} \lesssim  1 \, . $$
Similarly in the subcase $-|\xi| \le \tau \le 0$ we obtain $A=\max(1+\frac{r}{2},\frac{3}{2})-1-r = - \frac{r}{2}$ , $B= 1 - \max(1+\frac{r}{2},\frac{3}{2}) = -\frac{r}{2} \, ,$ which implies
$$ I \sim |\xi|^{\half} ||\tau|-|\xi||^{\half} |\xi|^{-\half} ||\tau|-|\xi||^{-\half} = 1 \, .$$
In the subcase $|\eta| + |\xi-\eta| \ge 2|\xi|$ we obtain by \cite{FK}, Lemma 4.4:
 \begin{align*}
&\int \delta(\tau - |\eta| + |\xi - \eta|) \, |\eta|^{-1-\frac{r}{2}} |\xi - \eta|^{-\frac{r}{2}} d\eta \\
&\sim ||\tau|-|\xi||^{-\half} ||\tau|+|\xi||^{-\half}\int_2^{\infty} (|\xi|x+\tau)^{-\frac{r}{2}} (|\xi|x-\tau)^{1-\frac{r}{2}}(x^2-1)^{-\half} dx \\
&\sim  ||\tau|-|\xi||^{-\half} ||\tau|+|\xi||^{-\half} \int_2^{\infty} (x+\frac{\tau}{|\xi|})^{-\frac{r}{2}} (x-\frac{\tau}{|\xi|})^{1-\frac{r}{2}} (x^2-1)^{-\half} dx \, \cdot|\xi|^{1-r} \, .
\end{align*}
We remark that in fact the lower limit of the integral can be chosen as 2 by inspection of the proof in \cite{FK}.
The integral converges, because $|\tau| \le |\xi|$ and $r > 1.$  This implies the bound
$$ I \lesssim |\xi|^{\half} ||\tau|-|\xi||^{\half-\frac{1}{2r}}||\tau|+|\xi||^{-\frac{1}{2r}} |\xi|^{\frac{1}{r}-1} \lesssim  1 \, . $$
Summarizing we obtain
$$\|q_{12}(u,v)\|_{H^r_{0,0}} \lesssim \|D^{\frac{1}{r}} u_0^{\pm_1}\|_{L^{r'}}  \| v_0^{\pm_2}\|_{L^{r'}} \, . $$
By the transfer principle Prop. \ref{Prop.0.1} we obtain the claimed result. 
\end{proof}

In a similar manner we can also estimate the nullform $q_{0j}(u,v)$ .
\begin{lemma}
\label{Lemma5.2}
Assume $0 \le \alpha_1,\alpha_2 $ ,  $\alpha_1+\alpha_2 \ge \frac{1}{r}$ and $ b > \frac{1}{r}$ . The following estimate applies
$$ \|q_{0j}(u,v)\|_{H^r_{0,0}} \lesssim \|u\|_{X^r_{\alpha_1,b,\pm_1}} \|v\|_{X^r_{\alpha_2,b,\pm_2}} \, . $$
\end{lemma}
\begin{proof}
	Again we may  reduce to the case $\alpha_1= \frac{1}{r}$ and $\alpha_2=0$ .
Arguing as in the proof of Lemma \ref{Lemma5.1} we use in the elliptic case the estimate (cf. \cite{FK}, Lemma 13.2):
$$|q_{0j}(\eta,\xi-\eta)| \lesssim \frac{(|\eta|+|\xi-\eta|-|\xi|)^{\half}}{\min(|\eta|^{\half},|\xi-\eta|^{\half})} \, . $$
In the case  $|\eta| \le |\xi-\eta|$ we obtain
\begin{align*}
I &= ||\tau|-|\xi||^{\half} \left( \int \delta(\tau - |\eta| - |\xi - \eta|) \, |\eta|^{-1-\frac{r}{2}} d\eta \right)^{\frac{1}{r}} \\
&\sim ||\tau|-|\xi||^{\half} |\tau|^{\frac{A}{r}}  ||\tau|-|\xi||^{\frac{B}{r}} = 1\, ,
\end{align*}
because $ A=\max(1+\frac{r}{2},\frac{3}{2}) - 1-\frac{r}{2} = 0$ and $B= 1-\max(1+\frac{r}{2},\frac{3}{2})-\frac{r}{2} = -\frac{r}{2}$ . \\
In the case $|\eta| \ge |\xi-\eta|$ we obtain
\begin{align*}
I &= ||\tau|-|\xi||^{\half} \left( \int \delta(\tau - |\eta| - |\xi - \eta|) \, |\eta|^{-1} |\xi-\eta|^{-\frac{r}{2}} d\eta \right)^{\frac{1}{r}} \\
&\sim ||\tau|-|\xi||^{\half} |\tau|^{\frac{A}{r}}  ||\tau|-|\xi||^{\frac{B}{r}} (1+\log \frac{|\tau|}{||\tau|-|\xi||})^{\frac{1}{r}}\, ,
\end{align*}
where  $A= \max(1,\frac{r}{2},\frac{3}{2})-1-\frac{r}{2} = \half-\frac{r}{2}$ and $B=-\half$, so that
$$I \lesssim ||\tau|-|\xi||^{\half} \tau^{\frac{1}{2r}-\half} ||\tau|-|\xi||^{-\frac  {1}{2r}} \lesssim  1 \, .$$
 In the hyperbolic case we obtain by \cite{FK}, Lemma 13.2:
$$|q_{0j}(\eta, \xi-\eta)| \lesssim |\xi|^{\half} \frac{(|\xi|-||\eta|-|\eta-\xi||)^{\half}}{|\eta|^{\half} |\xi-\eta|^{\half}} $$
and argue exactly as in the proof of Lemma \ref{Lemma1}. The proof is completed as before.
\end{proof}

We also need the same result for $q_0(u,v)$ .
\begin{lemma}
\label{Lemma5.3}
Assume $0 \le \alpha_1,\alpha_2 $ , $\alpha_1+\alpha_2 \ge \frac{1}{r}$ and $ b > \frac{1}{r}$ . The following estimate applies
$$ \|q_0(u,v)\|_{H^r_{0,0}} \lesssim \|u\|_{X^r_{\alpha_1,b,\pm_1}} \|v\|_{X^r_{\alpha_2,b,\pm_2}} \, . $$
\end{lemma}
\begin{proof}
	As before we reduce to the case $\alpha_1= \frac{1}{r}$ and $\alpha_2 =0$ .
We use in the elliptic case the estimate (cf. \cite{FK}, Lemma 13.2):
$$|q_{0}(\eta,\xi-\eta)| \lesssim \frac{|\eta|+|\xi-\eta|-|\xi|}{\min(|\eta|,|\xi-\eta|)} \, . $$
In the case $|\eta| \le |\xi-\eta|$ we have to estimate
\begin{align*}
I &= ||\tau|-|\xi|| \left( \int \delta(\tau - |\eta| - |\xi - \eta|) \, |\eta|^{-1-r}  d\eta \right)^{\frac{1}{r}} \\
&\sim ||\tau|-|\xi||  |\tau|^{\frac{A}{r}}  ||\tau|-|\xi||^{\frac{B}{r}}= 1 \, ,
\end{align*}
because $A=\max(1+r,\frac{3}{2})-1-r=0$ and $B=1-\max(1+r,\frac{3}{2})= -r$ . \\
In the case  $|\eta| \ge |\xi-\eta|$ we obtain
\begin{align*}
I &= ||\tau|-|\xi|| \left( \int \delta(\tau - |\eta| - |\xi - \eta|) \, |\eta|^{-1}  |\xi-\eta|^{-r} d\eta \right)^{\frac{1}{r}} \\
&\sim ||\tau|-|\xi||  |\tau|^{\frac{A}{r}}  ||\tau|-|\xi||^{\frac{B}{r}} \lesssim 1  \, ,
\end{align*} 
  because $A=\max(1,r,\frac{3}{2})-1-r= \half - r$ , $B=1-\frac{3}{2}=-\half$ and  $|\xi| \le |\tau|$ .
  
In the hyperbolic case we obtain by \cite{FK}, Lemma 13.2:
$$|q_{0j}(\eta, \xi-\eta)| \lesssim |\xi|\frac{|\xi|-||\eta|-|\eta-\xi||}{|\eta|\, |\xi-\eta|} \, .$$
In the subcase $|\eta|+|\xi-\eta| \le 2|\xi|$ we apply \cite{FK}, Prop. 4.5 and obtain
$$\int \delta(\tau - |\eta| + |\xi - \eta|) \, |\eta|^{-1-r} |\xi - \eta|^{-r} d\eta \sim |\xi|^A ||\xi|-|\tau||^B \ $$
where in the subcase $0 \le \tau \le |\xi|$ :  $A=\max(r,\frac{3}{2})-1-2r=\half-2r$ , $B=1-\max(r,\frac{3}{2})=-\half$ , so that
$$I \lesssim |\xi| \,||\tau|-|\xi|| \,|\xi|^{\frac{1}{2r}-2} ||\tau|-|\xi|^{-\frac{1}{2r}} \lesssim 1 \, , $$
whereas in the subcase $-|\xi| \le \tau \le 0$ we obtain $A=\max(r,\frac{3}{2})-1-2r = -r$ , $B=1-\max(1+r,\frac{3}{2}) = -r $ , so that $I \sim 1$ .

In the subcase $|\eta| + |\xi-\eta| \ge 2|\xi|$ we obtain by \cite{FK}, Lemma 4.4:
 \begin{align*}
&\int \delta(\tau - |\eta| + |\xi - \eta|) \, |\eta|^{-1-r} |\xi - \eta|^{-r} d\eta \\
&\sim ||\tau|-|\xi||^{-\half} ||\tau|+|\xi||^{-\half}\int_2^{\infty} (x+\frac{\tau}{|\xi|})^{-r} (x-\frac{\tau}{|\xi|})^{1-r} (x^2-1)^{-\half}dx \, \cdot|\xi|^{1-2r} \, .
\end{align*}
The integral converges, because $|\tau| \le |\xi|$ . This implies the bound
$$ I \lesssim |\xi| \,||\tau|-|\xi||\,||\tau|-|\xi||^{-\frac{1}{2r}} ||\tau|+|\xi||^{-\frac{1}{2r}} |\xi|^{\frac{1}{r}-2} \lesssim |\xi|^{\frac{1}{2r}-1} |\tau|-|\xi||^{1-\frac{1}{2r}} \lesssim 1 \, . $$
The proof is completed as the proof of Lemma \ref{Lemma1}.
\end{proof}

\begin{lemma}
\label{Lemma4}
Let $1 < r \le 2$ .
Assume $ \alpha_1,\alpha_2 \ge 0$ ,  $\alpha_1+\alpha_2 >\frac{3}{2r}$, $b_1 , b_2 > \frac{1}{2r}$, $b_1+b_2 > \frac{3}{2r}$. Then the following estimate applies:
$$\|uv\|_{H^r_{0,0}} \lesssim \|u\|_{X^r_{\alpha_1,b_1,\pm_1}} \|v\|_{X^r_{\alpha_2,b_2,\pm_2}} \, . $$
\end{lemma}
\begin{proof}
This follows from \cite{GT}, Prop. 3.1 by summation over the dyadic parts.
\end{proof}

\begin{lemma}
\label{Lemma5}
If $\alpha_1,\alpha_2,b_1,b_2 \ge 0$ , $\alpha_1+\alpha_2 > \frac{2}{r}$ and $b_1+b_2 > \frac{1}{r}$ the following estimate applies:
$$ \|uv\|_{H^r_{0,0}} \lesssim \|u\|_{H^r_{\alpha_1,b_1}} \|v\|_{H^r_{\alpha_2,b_2}} \, . $$
\end{lemma}
\begin{proof}
	We may assume $\alpha_1 = \frac{2}{r}+$ , $\alpha_2 =0$ , $b_1=\frac{1}{r}+$ , $b_2 =0$ (or similarly $b_1 =0$ , $b_2=\frac{1}{r}+$).
By Young's and H\"older's inequalities we obtain
\begin{align*}
\|uv\|_{H^r_{0,0}} &= \|\widehat{uv}\|_{L^{r'}_{\tau \xi}} \lesssim \|\widehat{u}\|_{L^{1}_{\tau \xi}}  \|\widehat{v}\|_{L^{r'}_{\tau\xi}}  \\
	& \lesssim \| \langle \xi \rangle^{-\frac{2}{r}-} \langle|\tau|-|\xi|\rangle^{-\frac{1}{r}-}\|_{L^r_{\tau\xi}}  \| \langle \xi \rangle^{\frac{2}{r}+} \langle|\tau|-|\xi|\rangle^{\frac{1}{r}+} \widehat{u}\|_{L^{r'}_{ \tau \xi}} \,  \|\widehat{v}\|_{L^{r'}_{\tau \xi}} \\
		& \lesssim \|u\|_{H^r_{\frac{2}{r}+,\frac{1}{r}+}} \|v\|_{H^r_{0,0}} \, .
		\end{align*}

\end{proof}

\begin{lemma}
\label{Lemma6}
Let $1 < r \le 2$ , $0 \le \alpha_1,\alpha_2$ and  $\alpha_1+\alpha_2 \ge \frac{1}{r}+b$ , $b >\frac{1}{r}$ . Then the following estimate applies:
$$ \|uv\|_{H^r_{0,b}} \lesssim
 \|u\|_{X^r_{\alpha_1,b, \pm_1}} \|v\|_{X^r_{\alpha_2,b,\pm_2}} \, . $$
\end{lemma}
\begin{proof}
	We may assume $\alpha_1 = \frac{1}{r}+b$ , $\alpha_2 =0$ .
We apply the "hyperbolic Leibniz rule" (cf. \cite{AFS}, p. 128):
\begin{equation}
\label{HLR}
 ||\tau|-|\xi|| \lesssim ||\rho|-|\eta|| + ||\tau - \rho| - |\xi-\eta|| + b_{\pm}(\xi,\eta) \, , 
\end{equation}
where
 $$ b_+(\xi,\eta) = |\eta| + |\xi-\eta| - |\xi| \quad , \quad b_-(\xi,\eta) = |\xi| - ||\eta|-|\xi-\eta|| \, . $$

Let us first consider the term $b_{\pm}(\xi,\eta)$ in (\ref{HLR}). Decomposing as before $uv=u_+v_++u_+v_-+u_-v_++u_-v_-$ , where $u_{\pm}(t)= e^{\pm itD} f , v_\pm(t) = e^{\pm itD} g$  , we use
$$ \widehat{u}_{\pm}(\tau,\xi) = c \delta(\tau \mp |\xi|) \widehat{f}(\xi) \quad , \quad  \widehat{v}_{\pm}(\tau,\xi) = c \delta(\tau \mp |\xi|) \widehat{g}(\xi)$$
and have to estimate
\begin{align*}
&\| \int  b^{b}_{\pm}(\xi,\eta) \delta(\tau - |\eta| \mp |\xi-\eta|) \widehat{f}(\xi) \widehat{g}(\xi-\eta)d\eta \|_{L^{r'}_{\tau \xi}} \\
& = \| \int ||\tau|-|\xi||^b \delta(\tau - |\eta| \mp |\xi-\eta|) \widehat{f}(\xi) \widehat{g}(\xi-\eta) d\eta \|_{L^{r'}_{\tau \xi}} \\
& \lesssim \sup_{\tau,\xi} I \, \|\widehat{D^{\frac{1}{r}+b} f}\|_{[L^{r'}}
\|\widehat{g}\|_{[L^{r'}} \, .
\end{align*}
Here we used H\"older's inequality, where
$$I = ||\tau|-|\xi||^b  (\int \delta(\tau-|\eta|\mp|\xi-\eta|) |\eta|^{-1-b r} d\eta)^{\frac{1}{r}} \, . $$
In order to obtain $I \lesssim 1$ we first consider the elliptic case $\pm_1=\pm_2=+$ and use \cite{FK}, Prop. 4.3. Thus 
$$ I \sim ||\tau|-|\xi||^b  \tau^{\frac{A}{r}} ||\tau|-|\xi||^{\frac{B}{r}} = ||\tau|-|\xi||^b ||\tau|-|\xi||^{-b} = 1$$
with $A=\max(1+br,\frac{3}{2})-(1+br)=0$ and $B=1-\max(1+br,\frac{3}{2})= -br$ .

Next we consider the hyperbolic case $\pm_1 = + \, , \, \pm_2=-$ . \\
First we assume $|\eta|+|\xi-\eta| \le 2 |\xi|$ and use \cite{FK}, Prop. 4.5 which gives
$$\int \delta(\tau - |\eta| + |\xi-\eta|) |\eta|^{-1-br}  d\eta \sim |\xi|^A ||\xi|-|\tau||^B \, , $$
where $A=\frac{3}{2}-(1+br)= \half-br$ , $B=1-\frac{3}{2}=- \half$ , if $0 \le \tau\le |\xi|$ ,so that
$$I \sim ||\tau|-|\xi||^b |\xi|^{\frac{1}{2r}-b} ||\tau|-|\xi||^{-\frac{1}{2r}} \lesssim 1 \, .$$   
If $-|\xi| \le \tau \le 0$ we obtain $A=\max(1+br,\frac{3}{2})-(1+br)=0$ , $B=1-\max(1+br,2)=-br$ , which implies  $I \lesssim 1$ .\\
Next we assume $|\eta|+|\xi-\eta| \ge 2|\xi|$ , use \cite{FK}, Lemma 4.4 and obtain
\begin{align*}
&I \sim ||\tau|-|\xi||^b (\int \delta(\tau-|\eta|-|\xi-\eta|) |\eta|^{-1-br} d\eta)^{\frac{1}{r}} \\
& \sim  ||\tau|-|\xi||^b (||\tau|-|\xi||^{-\half} ||\tau|+|\xi||^{-\half} \int_2^{\infty} (|\xi|x + \tau)^{-br} (|\xi|x-\tau)(x^2-1)^{-\half} dx)^{\frac{1}{r}} \\
& \sim ||\tau|-|\xi||^b (||\tau|-|\xi||^{-\half} ||\tau|+|\xi||^{-\half}\int_2^{\infty} (x+\frac{\tau}{|\xi|})^{-br} (x-\frac{\tau}{|\xi|}) (x^2-1)^{-\half} dx \,\cdot |\xi|^{1-br})^{\frac{1}{r}}  \, .
\end{align*}
This integral converges, because $\tau \le |\xi|$ and $b >\frac{1}{r}$ .This implies
$$I \lesssim  ||\tau|-|\xi||^{b-\frac{1}{2r}} ||\tau|+|\xi||^{-\frac{1}{2r}} |\xi|^{\frac{1}{r}-b} \lesssim  1 \, , $$
using $|\tau| \le |\xi|$ .

By the transfer principle we obtain
$$\| B_{\pm}^b (u,v)\|_{X^r_{0,0}} \lesssim \|u\|_{X^r_{\frac{1}{r}+b,b,\pm_1}} \|v\|_{X^r_{0,b,\pm_2}} \, .$$ 
 Here $B^b_{\pm}$ denotes the operator with Fourier symbol $b_{\pm}$ . \\
Consider now the term $||\rho|-|\eta||$ (or similarly $||\tau-\rho|-|\xi-\eta||$) in (\ref{HLR}). We have to prove
$$ \|u D_-^b v\|_{H^r_{0,0}} \lesssim \|u\|_{X^r_{\alpha_1,b,\pm_1}} \|v\|_{X^r_{\alpha_2,b,\pm_2}} \, , $$
which is implied by
$$\|uv\|_{H^r_{0,0}} \lesssim \|u\|_{X^r_{\alpha_1,b,\pm_1}} \|v\|_{X^r_{\alpha_2,0,\pm_2}} \, . $$
This results from Lemma \ref{Lemma5}, because $\alpha_1+\alpha_2\ge \frac{1}{r} +b > \frac{2}{r}$ , which completes the proof. 
\end{proof}

\section{ Proof of (\ref{21}) - (\ref{37}) in the case $r=1+$:}

\begin{proof}
The estimates are proven by the results of chapter 5.\\[0.5em]
{\bf Assumption:} $ s > 1+\frac{1}{2r}$ , $l \ge \half$ , $s-1 \le l \le s$ , $2s-l>1+\frac{1}{r}$ , $2l-s+2 > \frac{3}{2r}$ and $b=\frac{1}{r}+$ .  \\[0.5em]
{\bf Proof of (\ref{21}) and (\ref{22}):} This reduces to
$$\|q(u,v)\|_{H^r_{s-1,0}} \lesssim \|u\|_{X^r_{s,b}} \|u\|_{X^r_{s-1,b}} \, . $$
By the fractional Leibniz rule this results from Lemma \ref{Lemma5.1} or Lemma \ref{Lemma5.2} for $ s > \frac{1}{r}$. \\ 
{\bf Proof of (\ref{23}):} 
This reduces to
$$\|q(u,v)\|_{H^r_{l-1,0}} \lesssim \|u\|_{X^r_{s,b}} \|v\|_{X^r_{l-1,b}} \, . $$
Let us first consider the case $|\xi-\eta| \le 1$. It suffices to show
$$\|Q(u,v)\|_{H^r_{l,0}} \lesssim \|u\|_{X^r_{s+1,b}} \|v\|_{X^r_{N,b}}$$
for any $N \in \mathbb N$ . By Lemma \ref{Lemma5} we obtain easily:
$$\|Q(u,v)\|_{H^r_{l-1,0}} \lesssim \| \Lambda u \Lambda v\|_{H^r_{l-1,0}} \lesssim \|\Lambda u\|_{X^r_{s,b}} \|\Lambda v\|_{X^r_{N,b}} \le \|u\|_{X^r_{s+1,b}} \|v\|_{X^r_{l,b}}$$
for sufficiently large $N$ . 

From now on we assume $|\xi-\eta| \ge 1$ . 

In the elliptic case we use the estimate (cf. \cite{FK},Lemma 13.2):
$$|q(\eta,\xi-\eta)| \lesssim \frac{(|\eta| + |\xi-\eta|-|\xi|)^{\half}}{\min(|\eta|,|\xi-\eta|)^{\half}} \, .$$  We argue as in Lemma \ref{Lemma5.1} and Lemma \ref{Lemma5.2}.
In the subcase $|\xi-\eta| \lesssim |\eta|$ we estimate for $1 \ge l \ge \half$ and $s=\half+\frac{1}{r}$ :
\begin{align*}
I &= ||\tau|-|\xi||^{\half} (\int \delta(\tau-|\eta|-|\xi-\eta|) |\eta|^{-sr}  d\eta)^{\frac{1}{r}} \\
	& \sim ||\tau|-|\xi||^{\half} |\tau|^{\frac{A}{r}} ||\tau|-|\xi||^{\frac{B}{r}} = 1 \, ,
	\end{align*}
where by \cite{FK}, Prop. 4.3 we obtain $A=\max(sr,\frac{3}{2})-sr=0$ , $B= 1-\max(sr,\frac{3}{2}) = -\frac{r}{2}$ , which implies
$$\|q(u,v)\|_{H^r_{l-1,0}} \le \|q(u,v)\|_{H^r_{0,0}} \le \|u\|_{X^r_{\half+\frac{1}{r},b}} \|v\|_{X^r_{-\half,b}} \le \|u\|_{X^r_{s,b}} \|v\|_{X^r_{l-1,b}} \, . $$
If $l \ge 1$ we apply the fractional Leibniz and reduce to Lemma \ref{Lemma5.1} or Lemma \ref{Lemma5.2} using $s >\frac{1}{r}$ . \\
In the subcase $|\eta| \ll |\xi-\eta| \sim |\xi|$ we assume $s= \frac{1}{r}$ and obtain similarly
\begin{align*}
I =& |\xi|^{l-1} ||\tau|-|\xi||^{\half} (\int \delta(\tau-|\eta|-|\xi-\eta|) |\eta|^{-(s+\half)r} |\xi-\eta|^{(1-l)r} d\eta)^{\frac{1}{r}} \\
& \sim ||\tau|-|\xi||^{\half} (\int \delta(\tau-|\eta|-|\xi-\eta|) |\eta|^{-(s+\half)r}  d\eta)^{\frac{1}{r}} \\
& \sim ||\tau|-|\xi||^{\half} |\tau|^{\frac{A}{r}} ||\tau|-|\xi||^{\frac{B}{r}} 
 = 1 \, ,
\end{align*}
where $A=\max((s+\half)r,\frac{3}{2})-(s+\half)r=0$ , $B=1-\max((s+\half)r,\frac{3}{2}) = - \frac{r}{2}$ . This implies the claimed estimate.

In the hyperbolic case we use (cf. \cite{FK}, Lemma 13.2):
$$|q(\eta,\xi-\eta)| \lesssim \frac{|\xi|^{\half}(|\xi| - ||\eta|-|\xi-\eta||)^{\half}}{(|\eta|\,|\xi-\eta|)^{\half}} \, .$$ 
Now by an elementary calculation (cf. \cite{AFS1}) we obtain
$$(||\xi|-||\eta|-|\xi-\eta||)^{\half} \lesssim ||\tau|-|\xi||^{\half} + |\lambda+|\eta||^{\half} + |\lambda - \tau+|\xi-\eta||^{\half} \, . $$
Thus we reduce to the estimates
\begin{align*}
\|uv\|_{H^r_{l-\half,\half}} & \lesssim \|u\|_{X^r_{s+\half,b}} \|v\|_{X^r_{l-\half,b}} \, , \\
\|uv\|_{H^r_{l-\half,0}} & \lesssim \|u\|_{X^r_{s+\half,b-\half}} \|v\|_{X^r_{l-\half,b}} \, , \\
\|uv\|_{H^r_{l-\half,0}} & \lesssim \|u\|_{X^r_{s+\half,b}} \|v\|_{X^r_{l-\half,b-\half}} \, . 
\end{align*}
For $l \ge \half$ , $s\ge 1+\frac{1}{2r}$ the first estimate follows from Lemma \ref{Lemma6} by the fractional Leibniz rule, because $s+\half \ge \frac{3}{2} + \frac{1}{2r} > \frac{1}{r}+b$ for $b =\frac{1}{r}+$ . The other two estimates follow from Lemma \ref{Lemma5}. The proof of (\ref{23}) is complete.\\
{\bf Proof of (\ref{24}) and (\ref{25}):} Concerning (\ref{24}) we have to show
$$\|q(u,v)\|_{H^r_{l-1,0}} \lesssim \|u\|_{X^r_{s-1,b}} \|u\|_{X^r_{s-1,b}} \, . $$
This is implied by Lemma \ref{Lemma5.1} or Lemma \ref{Lemma5.2}, if $l \le 1$ and $s > 1+\frac{1}{2r}$ or more generally $l \le s$ and $2s-2-(l-1) >\frac{1}{r} \, \Leftrightarrow 2s-l > 1+\frac{1}{r}$ . Similarly (\ref{25}) follows from Lemma \ref{Lemma5.3}.\\
{\bf Proof of (\ref{26}):} By Lemma \ref{Lemma2.1} we have to prove
$$\|q_{12}(u,\Lambda v)\|_{X^r_{s-1,0}} \lesssim \|u\|_{X^r_{s,b}} \|v\|_{X^r_{s,b}} \, $$
which by the fractional Leibniz rule results from Lemma \ref{Lemma5.1}. Moreover we need
$$\|q_{0}(u,\Lambda v)\|_{X^r_{s-1,0}} \lesssim \|u\|_{X^r_{s,b}} \|v\|_{X^r_{s,b}} \, $$
which is given by Lemma \ref{Lemma5.3}. Finally
$$ \| (\Lambda^{-2}u)v\|_{X^r_{s-1,0}} + \| u(\Lambda^{-2}v)\|_{X^r_{s-1,0}} \lesssim \|u\|_{X^r_{s,b}} \|v\|_{X^r_{s-1,b}}$$
by Lemma \ref{Lemma5} for $ s+2 > \frac{2}{r}$ , which is fulfilled. \\
{\bf Proof of (\ref{27}) and (\ref{28}):} The estimates result from Lemma \ref{Lemma5}, if $s \ge 1$. \\
{\bf Proof of (\ref{29}):} We reduce to
$$\|uv\|_{H^r_{s-1,0}} \lesssim \|u\|_{X^r_{l+1,b}} \|v\|_{X^r_{l,b}} \, . $$
By the fractional Leibniz rule this is implied by Lemma \ref{Lemma4}, if $l\ge s-1$ and $2l+1-(s-1)= 2l-s+2 >\frac{3}{2r}$. This is one of our assumptions. It is fulfilled for $l=\half$ and $s= 1+\frac{1}{2r}+$ .\\
{\bf Proof of (\ref{30}):} The estimate reduces to
$$\|uv\|_{H^r_{l-1,0}} \lesssim \|u\|_{X^r_{s+2,b}} \|v\|_{X^r_{l-1,b}} \, . $$
If $l \ge 1$ this easily follows from Lemma \ref{Lemma5}. Assume from now on $0 \le l < 1$. The result is by duality equivalent to
$$\|uw\|_{H^{r'}_{1-l,-b}} \lesssim \|u\|_{H^r_{s+2,b}} \|w\|_{H^{r'}_{1-l,0}} \, , $$
which by the fractional Leibniz rule reduces to the estimates
$$\|uw\|_{H^{r'}_{0,0}} \lesssim \|u\|_{H^r_{s+2,b}} \|w\|_{H^{r'}_{0,0}}$$
and 
$$\|uw\|_{H^{r'}_{0,0}} \lesssim \|u\|_{H^r_{s+l+1,b}} \|w\|_{H^{r'}_{0,0}} \, . $$
Now we obtain
\begin{align*}
\|uw\|_{H^{r'}_{0,0}} & = \|\widehat{uw}\|_{L^r_{\tau \xi}} \le \|\widehat{u}\|_{L^1_{\tau \xi}} \|\widehat{w}\|_{L^r_{\tau \xi}} \\
& \lesssim \|\langle \xi \rangle^{-\frac{2}{r}-} \langle |\tau|-|\xi|\rangle^{-\frac{1}{r}-}\|_{L^r_{\tau\xi}} \|\langle \xi \rangle^{\frac{2}{r}+} \langle |\tau|-|\xi|\rangle^{\frac{1}{r}+} \widehat{u}\|_{L^{r'}_{\tau \xi}} \|\widehat{w}\|_{L^r_{\tau \xi}} \\
& \lesssim\|u\|_{H^r_{\frac{2}{r}+,\frac{1}{r}+}} \|w\|_{H^{r'}_{0,0}} \lesssim \|u\|_{H^r_{s+1,b}} \|w\|_{H^{r'}_{0,0}} \, ,
\end{align*}if $s > \frac{2}{r}-1$ , which is fulfilled.\\
{\bf Proof of (\ref{31}):} We reduce to
$$\|uv\|_{H^r_{l-1,0}} \lesssim \|u\|_{X^r_{s+1,b}} \|v\|_{X^r_{s-1,b}} \, , $$
which results from Lemma \ref{Lemma5}, if $l \le 1$ and  $s \ge 1$ or $l \ge 1$ and $2s-l > \frac{2}{r}-1$ and moreover $s\ge l$ . \\
{\bf Proof of (\ref{32}):} The estimate
$$\|uvw\|_{H^r_{s-1,0}} \lesssim \|u\|_{X^r_{l+1,b}} \|v\|_{X^r_{s,b}}  \|w\|_{X^r_{s,b}}$$
reduces by the fractional Leibniz rule to the estimates
$$\|uvw\|_{X^r_{0,0}} \lesssim \|u\|_{X^r_{l-s+2,b}} \|v\|_{X^r_{s,b}}  \|w\|_{X^r_{s,b}}$$
and
$$\|uvw\|_{X^r_{0,0}} \lesssim \|u\|_{X^r_{l+1,b}} \|v\|_{X^r_{1,b}}  \|w\|_{X^r_{s,b}} \, .$$
Now we obtain by Lemma \ref{Lemma5} :
\begin{align*}
\|uvw\|_{X^r_{0,0}} &\lesssim \|u\|_{X^r_{l-s+2,b}} \|vw\|_{X^r_{s-l-2+\frac{2}{r}+,0}} \\
& \lesssim \|u\|_{X^r_{l-s+2,b}} \|v\|_{X^r_{s,b}}  \|w\|_{X^r_{s,b}} \, .
\end{align*}
The last estimate results from Lemma \ref{Lemma5}, because $2s-(s-l-2+\frac{2}{r}) = s+l-\frac{2}{r}+2 > 1+\frac{1}{2r} + \frac{1}{2r}-\frac{2}{r}+2=3-\frac{1}{r} > \frac{2}{r}$ . Moreover in exactly the same way
we obtain
\begin{align*}
\|uvw\|_{X^r_{0,0}} &\lesssim \|u\|_{X^r_{l+1,b}} \|vw\|_{X^r_{\frac{2}{r}-l-1+,0}} \\
& \lesssim \|u\|_{X^r_{l+1,b}} \|v\|_{X^r_{1,b}}  \|w\|_{X^r_{s,b}} \, .
\end{align*}
{\bf Proof of (\ref{33}):} We apply Lemma \ref{Lemma5} which implies
$$\|u \Lambda^{-1}(vw)\|_{H^r_{s-1,0}} \lesssim \|u\|_{X^r_{l,b}} \|\Lambda^{-1}(vw)\|_{X^r_{\frac{2}{r}+s-1-l+,0}} \lesssim \|u\|_{X^r_{l,b}} \|v\|_{X^r_{s,b}} \|w\|_{X^r_{s,b}} \, ,$$
because $2s- (\frac{2}{r}-l+s-2) > 1+\frac{1}{2r}+\frac{1}{2r}-\frac{2}{r}+1 = 3-\frac{1}{r} > \frac{2}{r}$  and $l\ge s-1$ . \\
{\bf Proof of (\ref{35}):} Lemma \ref{Lemma5} implies
$$\|u vw\|_{H^r_{s-1,0}} \lesssim \|u\|_{X^r_{s,b}} \|vw\|_{X^r_{\frac{2}{r}-1+,0}} \lesssim \|u\|_{X^r_{s,b}} \|v\|_{X^r_{s,b}} \|w\|_{X^r_{s,b}} \, , $$
because $2s-\frac{2}{r}+1 > 2+\frac{1}{r}-\frac{2}{r}+1 >\frac{2}{r}$ . \\
{\bf Proof of (\ref{36}):} For $l\le 1$ we obtain by Lemma \ref{Lemma5}:
$$\|u vw\|_{H^r_{l-1,0}} \lesssim  \|u vw\|_{H^r_{0,0}} \lesssim \|w\|_{X^r_{l,b}} \|uv\|_{X^r_{\frac{2}{r}-l+,0}} \lesssim \|w\|_{X^r_{l,b}} \|u\|_{X^r_{s,b}} \|v\|_{X^r_{s,b}} \,  $$
For the last estimate we applied Lemma \ref{Lemma4}, where we used
$2s-(\frac{2}{r}-l) > 2+\frac{1}{r} -\frac{2}{r}+\frac{1}{2r} = 2-\frac{1}{2r} > \frac{3}{2r}$ .\\
If $l \ge 1$ we use the fractional Leibniz rule which reduces the claimed estimate to
$$
\|uvw\|_{H^r_{0,0}} \lesssim \|u\|_{H^r_{1,b}} \|vw\|_{H^r_{\frac{2}{r}-1+,0}} 
 \lesssim \|u\|_{H^r_{1,b}} \|v\|_{H^r_{s,b}} \|w\|_{H^r_{s,b}} \,  
$$
by applying Lemma \ref{Lemma5} twice, where we used $2s-(\frac{2}{r}-1) > \frac{2}{r}$ .
Moreover
 $$
\|uvw\|_{H^r_{0,0}} \lesssim \|u\|_{H^r_{l,b}} \|vw\|_{H^r_{\frac{2}{r}-l+,0}} 
 \lesssim \|u\|_{H^r_{l,b}} \|v\|_{H^r_{s-(l-1),b}} \|w\|_{H^r_{s,b}} \,  
$$
by Lemma \ref{Lemma5} using $2s-(l-1)-(\frac{2}{r}-l) > \frac{2}{r}$ .\\
{\bf Proof of (\ref{34}):} We obtain
\begin{align*}
\|\Lambda^{-1}(uv)wz\|_{H^r_{s-1,0}} & \lesssim \|\Lambda^{-1}(uv)\|_{H^r_{\frac{7}{2r}-2+,b}} \|wz\|_{H^r_{s-(\frac{3}{2r}-1),0}} \\
& \lesssim \|u\|_{X^r_{s,b}} \|v\|_{X^r_{s,b}} \|w\|_{X^r_{s,b}} \|z\|_{X^r_{s,b}} \, .
\end{align*}
For the first step we applied Lemma \ref{Lemma5} using $s-(\frac{3}{2r}-1)+\frac{7}{2r}-2-(s-1) = \frac{2}{r}$ . For the last estimate we apply Lemma \ref{Lemma6} using $2s-(\frac{7}{2r}-3) > 2+\frac{1}{r} -\frac{7}{2r}+3=5-\frac{5}{2r} > \frac{1}{r}+b$ and also Lemma \ref{Lemma5} using $2s-s+\frac{3}{2r}-1 > 1+\frac{1}{2r}+\frac{3}{2r}-1 = \frac{2}{r}$ .\\
{\bf Proof of (\ref{37}):} For $l \le 1$ we use Lemma \ref{Lemma5} and Lemma \ref{Lemma6} :
\begin{align*}
\|uvwz\|_{H^r_{l-1,0}}\lesssim \|uvwz\|_{H^r_{0,0}} & \lesssim \|uv\|_{H^r_{\frac{1}{r}+,b}} \|wz\|_{H^r_{\frac{1}{r},0}} \\
& \lesssim \|u\|_{X^r_{s,b}} \|v\|_{X^r_{s,b}} \|w\|_{X^r_{s,b}}\|z\|_{X^r_{s,b}} \, ,
\end{align*}
where we used $2s-\frac{1}{r} > 2 > \frac{1}{r}+b> \frac{2}{r} $ for $b=\frac{1}{r}+$ . For $l>1$ we use the fractional Leibniz rule  which reduces the claimed estimate to
\begin{align*}
 \|uvwz\|_{H^r_{0,0}} & \lesssim \|uv\|_{H^r_{\frac{1}{r}+,b}} \|wz\|_{H^r_{\frac{1}{r},0}} \\
& \lesssim \|u\|_{X^r_{s,b}} \|v\|_{X^r_{s,b}} \|w\|_{X^r_{s-(l-1),b}}\|z\|_{X^r_{s,b}} \, ,
\end{align*}
provided $2s-(l-1)-\frac{1}{r} > \frac{2}{r} \Leftrightarrow 2s-l > \frac{3}{r}-1$ and $2s-\frac{1}{r} >2>\frac{1}{r}+b$ for $b=\frac{1}{r}+$.	
	\end{proof}

\section{Proof of (\ref{21}) - (\ref{37}) in the case $r=2$.} 
We recall the null forms given by
\begin{align*}
Q_0(u,v) & = - \partial_t u \partial_t v + \partial^i u \partial_i v \\
Q_{0i}(u,v) & = \partial_t \partial_i v - \partial_i u \partial_t v \\
Q_{12}(u,v) & = \partial_1 u \partial_2 v - \partial_2 u \partial_1 v \, 
\end{align*}
 and the substitution 
$$ u=u_++u_- \, , \, \partial_t u= i \Lambda(u_+-u_-) \, , \, v=v_++v_- \, , \, \partial_t v = i \Lambda(v_+-v_-) \, . $$
We consider the Fourier symbols
\begin{align*}
q_0(\xi,\eta) & = \langle \xi \rangle \langle \eta \rangle - \langle \xi,\eta \rangle = |\xi||\eta|(1-\frac{\langle \xi, \eta \rangle}{|\xi||\eta|}) + \langle \xi \rangle \langle \eta \rangle - |\xi||\eta| \\
q_{0i}(\xi,\eta) & = - \langle \xi \rangle \eta_i + \xi_i \langle \eta \rangle = |\xi||\eta|(\frac{\xi_i}{|\xi|} - \frac{\eta_i}{|\eta|}) + \xi_i(\langle \eta \rangle - |\eta|) -(\langle \xi \rangle - |\xi|) \eta_i\\
q_{12}(\xi,\eta) & = -\xi_1 \eta_2 + \xi_2 \eta_1 \, .
\end{align*}
Then the Fourier symbols of $Q_0(u,v)$ , $Q_{0i}(u,v)$ and $Q_{12}(u,v)$ are linear combinations of $q_0(\pm_1 \xi,\pm_2 \eta)$ , $q_{0i}(\pm_1 \xi,\pm_2 \eta)$ and $q_{12}(\pm_1 \xi,\pm_2 \eta)$, respectively.

The following simple observation can be found e.g. in \cite{ST}.
	\begin{lemma}
		\label{Lemma1}
		The Fourier symbols satisfy
		\begin{align*}
		|q_0(\pm_1\xi,\pm_2\eta)| & \lesssim |\xi||\eta| \angle{(\pm_1\xi,\pm_2\eta)}^2 + \frac{1}{\min(\langle \xi \rangle,\langle \eta \rangle)} \\
		|q_{0i}(\pm_1\xi,\pm_2\eta)| & \lesssim |\xi||\eta| \angle{(\pm_1\xi,\pm_2\eta)} + \frac{|\xi|}{\langle \eta \rangle} + \frac{|\eta|}{\langle \xi \rangle}  \\
			|q_{12}(\pm_1\xi,\pm_2\eta)| & \lesssim |\xi||\eta| \angle{(\pm_1\xi,\pm_2\eta)}  
		\end{align*}
	\end{lemma}
The estimate for the angle in the following Lemma was proven in \cite{AFS1}, Lemma 5:
\begin{lemma}
	\label{Lemma2}
	Let $\alpha,\beta,\gamma \in [0,\half]$ , $\tau,\lambda \in \R$ , $ \xi,\eta \in \R^2 $, $\xi,\eta \neq 0$ . Then the following estimate applies for all signs $\pm_1,\pm_2$ :
$$\angle{(\pm_1 \xi, \pm_2 \eta)} \lesssim \left(\frac{\langle|\tau+\lambda|-|\xi+\eta|\rangle}{\min(\langle\xi \rangle,\langle \eta \rangle)}\right)^{\alpha} +\left(\frac{\langle -\tau \pm_1|\xi|\rangle}{\min(\langle\xi\rangle,\langle \eta \rangle)}\right)^{\beta}
+\left(\frac{\langle -\lambda \pm_2|\eta|\rangle}{\min(\langle\xi \rangle,\langle \eta \rangle)}\right)^{\gamma} \, .
 $$
\end{lemma}

The following bilinear estimates for wave-Sobolev spaces were proven in \cite{AFS}, Lemma 7.
\begin{prop}
\label{Prop.7.1}
	Let $s_0,s_1,s_2 \in \R$ , $b_0,b_1,b_2 \ge 0$ . Assume that
	 \begin{align*}
	 b_0+b_1+b_2 &> \half \\
	 s_0+s_1+s_2 &>\frac{3}{2}-(b_0+b_1+b_2) \\
	 s_0+s_1+s_2 &> 1 - \min_{i \neq j} (b_i+b_j) \\
	s_0+s_1+s_2 &> \half - \min_i b_i  \\
	s_0+s_1+s_2& > 1-\min(b_0+s_1+s_2,s_0+b_1+s_2,s_0+s_1+b_2) \\
	s_0+s_1+s_2 & \ge \frac{3}{4} \\
	\min_{i \neq j} (s_i+s_j) & \ge 0 \, ,
	\end{align*}
where the last two inequalities are not both equalities. Then the following estimate applies:
$$ \|uv\|_{H^{-s_0,-b_0}} \lesssim \|u\|_{H^{s_1,b_1}} \|v\|_{H^{s_2,b_2}} \, .$$
If $b_0 <0$ , this remains true provided we  additionally assume $b_0+b_1 > 0$ , $b_0+b_2 > 0$ and $s_1+s_2 > -b_0$ .
\end{prop}

\begin{Cor}
\label{Cor.1}
If $b_0,b_1,b_2 \ge 0$ , $b_0+b_1+b_2 > \half$ and $\min_{i \neq j}(s_i+s_j) \ge 0$ the assumption $ s_0+s_1+s_2 > 1$ is sufficient.
\end{Cor}

\begin{Cor}
	\label{Cor.2}
	If $b_0 \ge 0$ , $b_1,b_2 > \half$ the following assumptions are sufficient:
	$$s_0+s_1+s_2 > 1-(b_0+s_1+s_2) \, , \, s_0+s_1+s_2 \ge \frac{3}{4} \,, \, \min_{i \neq j}(s_i+s_j) \ge 0 \, ,$$  where the last two inequalities are not both equalities.
\end{Cor}

Now we are ready to prove the inequalities (\ref{21}) - (\ref{37}) in the case $r=2$. \\[0.5em]
{\bf Assumption:} $s>\frac{3}{4}$ , $l>-\frac{1}{4}$ , $s\ge l\ge s-1$ , $2s-l > \frac{5}{4}$ , $4s-l > 3$ , $ 3s-2l > \frac{3}{2}$, $2l-s > - \frac{5}{4}$ . \\[0.5em]
{\bf Proof of (\ref{21}) and (\ref{22}):} We have to prove
$$ \|Q(u,v)\|_{H^{s-1,-\half+2\epsilon}} \lesssim \|u\|_{X^{s+1,\half+\epsilon}} \|v\|_{X^{s,\half+\epsilon}} \, . $$ 
By Lemma \ref{Lemma1} and Lemma \ref{Lemma2} we may reduce to the following estimates:
\begin{align*}
\|uv\|_{H^{s-1,0}} & \lesssim \|u\|_{H^{s+\half-2\epsilon,\half+\epsilon}} \|v\|_{H^{s-1,\half+\epsilon}}\\
\|uv\|_{H^{s-1,0}} & \lesssim \|u\|_{H^{s,\half+\epsilon}} \|v\|_{H^{s-\half-2\epsilon,\half+\epsilon}} \\
\|uv\|_{H^{s-1,-\half +2\epsilon}} & \lesssim \|u\|_{H^{s+\half,\epsilon}} \|v\|_{H^{s-1,\half+\epsilon}}\\
\|uv\|_{H^{s-1,-\half+2\epsilon}} & \lesssim \|u\|_{H^{s,\epsilon}} \|v\|_{H^{s-\half,\half+\epsilon}} \\
\|uv\|_{H^{s-1,-\half+2\epsilon}} & \lesssim \|u\|_{H^{s+\half,\half+\epsilon}} \|v\|_{H^{s-1,\epsilon}} \\
\|uv\|_{H^{s-1,-\half+2\epsilon}} & \lesssim \|u\|_{H^{s,\half+\epsilon}} \|v\|_{H^{s-\half,\epsilon}} \, .
\end{align*}
By Cor. \ref{Cor.1} these  estimates are fulfilled for a sufficiently small $\epsilon > 0$ , if $ s > \half $ .
In the case $Q=Q_{0i}$ we additionally have to show
\begin{align*}
\|uv\|_{H^{s-1,-\half+2\epsilon}} & \lesssim \|u\|_{H^{s+2,\half+\epsilon}} \|v\|_{H^{s-1,\half+\epsilon}} \, ,\\
\|uv\|_{H^{s-1,-\half+2\epsilon}} & \lesssim \|u\|_{H^{s,\half+\epsilon}} \|v\|_{H^{s+1,\half+\epsilon}} \, ,
\end{align*}
which are also fulfilled by Cor. \ref{Cor.1} . \\
{\bf Proof of (\ref{23}):}
We need $$ \|Q(u,v)\|_{H^{l-1,-\half+2\epsilon}} \lesssim \|u\|_{X^{s+1,\half+\epsilon}} \|v\|_{X^{l,\half+\epsilon}} \, . $$ 
Using Lemma \ref{Lemma1} and Lemma \ref{Lemma2} we reduce to six estimates as above. A typical one is
$$\|uv\|_{H^{l-1,0}} \lesssim \|u\|_{H^{s+\half-2\epsilon,\half+\epsilon}} \|v\|_{H^{l-1,\half+\epsilon}} \, , $$
which by Cor. \ref{Cor.1} is fulfilled if $s > \half$ and $s>l-\frac{3}{2}$ .  The other estimates may be handled similarly. \\
{\bf Proof of (\ref{24}) and (\ref{25}):}
We have to prove
 $$ \|Q(u,v)\|_{H^{l-1,-\half+2\epsilon}} \lesssim \|u\|_{X^{s,\half+\epsilon}} \|v\|_{X^{s,\half+\epsilon}} \, . $$ 
 This reduces to the following estimates by Lemma \ref{Lemma1} and Lemma \ref{Lemma2}:
 \begin{align*}
 \|uv\|_{H^{l-1,0}} &\lesssim \|u\|_{H^{s-1,\half+\epsilon}} \|v\|_{H^{s-\half-2\epsilon,\half+\epsilon}} \, , \\
 \|uv\|_{H^{l-1,-\half+2\epsilon}} &\lesssim \|u\|_{H^{s-1,\epsilon}} \|v\|_{H^{s-\half,\half+\epsilon}} \, , \\
 \|uv\|_{H^{l-1,-\half+2\epsilon}} &\lesssim \|u\|_{H^{s-1,\half+\epsilon}} \|v\|_{H^{s-\half,\epsilon}} \, .
 \end{align*}
 For the first estimate we use Cor. \ref{Cor.2}. This requires the conditions $s \ge l$ and
 \begin{align*}
 (s-1)+(s-\half) > 0 & \Leftrightarrow \, s > \frac{3}{4} \, , \\
 (1-l)+(s-1)+(s-\half) > \frac{3}{4} & \Leftrightarrow \, 2s-l > \frac{5}{4} \, , \\
 (1-l)+(s-1) +(s-\half) > 1-((s-1)+(s-\half)) > 0 & \Leftrightarrow 4s-l > 3 \, , \\
 (1-l)+(s-1) \ge 0 & \Leftrightarrow \, s \ge l \, .
 \end{align*}
 The second estimate is by duality equivalent to
 $$\|vw\|_{H^{1-s,-\epsilon}} \lesssim \|v\|_{H^{s-\half,\half+\epsilon}} \|w\|_{H^{1-l,\half-2\epsilon}} \, , $$
 which by Cor. \ref{Cor.2} requires $2s-l > \frac{5}{4}$ and moreover
 	$$2s-l-\half > 1-((s-\half)+(1-l)) \, \Leftrightarrow \, 3s-2l > 1 \, . $$
 	The third estimate is by duality equivalent to
  $$\|uw\|_{H^{\half-s,-\epsilon}} \lesssim \|u\|_{H^{s-1,\half+\epsilon}} \|w\|_{H^{1-l,\half-2\epsilon}} \, , $$
 which moreover  requires
 $$2s-l-\half > 1-((s-1)+(1-l)) \, \Leftrightarrow \, 3s-2l > \frac{3}{2} \, . $$
 In the case of $Q=Q_{0i}$ and $Q=Q_0$ we also need
 $$\|uv\|_{H^{l-1,-\half+2\epsilon}} \lesssim \|u\|_{H^{s+1,\half+\epsilon}} \|v\|_{H^{s-1,\half+\epsilon}} \, , $$
 which by Cor. \ref{Cor.1} is fulfilled for $s \ge l$ and
 $$(s+1)+(s-1)+(1-l) > 1 \, \Leftrightarrow \, 2s-l > 0 \, .$$
 {\bf Proof of (\ref{26}):}
 We have to prove
 $$\Gamma^1(u,v)\|_{H^{s-1,-\half+2\epsilon}} \lesssim \|u\|_{H^{s,\half+\epsilon}} \|v\|_{H^{s-1,\half+\epsilon}}\, . $$
 Now we use Lemma \ref{Lemma2.1} :
 $$ \Gamma^1(u,v) \preceq   Q_{12}(D^{-1}u,D^{-1}v) + Q_0(D^{-1}u,D^{-1}v) + (\Lambda^{-2} u)v + u (\Lambda^{-1}v) \, . $$ 
 Thus we need
 $$\|Q(u,v)\|_{H^{s-1,-\half+2\epsilon}} \lesssim \|u\|_{H^{s+1,\half+\epsilon}}\|v\|_{H^{s,\half+\epsilon}} \, , $$
 which is the same estimate, which was treated already in (\ref{21}) and (\ref{22}). Moreover we need the following estimates:
 \begin{align*}
 \|uv\|_{H^{s-1,-\half+2\epsilon}} &\lesssim \|u\|_{H^{s+2,\half+\epsilon}} \|v\|_{H^{s-1,\half+\epsilon}}\, ,\\
 \|uv\|_{H^{s-1,-\half+2\epsilon}} &\lesssim \|u\|_{H^{s,\half+\epsilon}} \|v\|_{H^{s+1,\half+\epsilon}}\, ,
\end{align*}
 which easily follow from Cor. \ref{Cor.1}. \\
 {\bf Proof of (\ref{27}),(\ref{28}) and (\ref{30}):}
These estimates result from Cor. \ref{Cor.1}. \\
{\bf Proof of (\ref{29}):}
This reduces to
$$\|uv\|_{H^{s-1,-\half+2\epsilon}} \lesssim \|u\|_{H^{l+1,\half+\epsilon}} \|v\|_{H^{l,\half+\epsilon}} \, . $$
By Cor. \ref{Cor.2} this requires
$1-s+l+1+l > \frac{3}{4} \, \Leftrightarrow \, 2l-s > - \frac{5}{4}$ and
$1-s+l+1+l > 1-(\half+2l+1) \, \Leftrightarrow \,4l-s > -\frac{5}{2}$
 and $l \ge s-1$ , which hold by our assumptions. \\
 {\bf Proof of (\ref{31}):}
 This reduces to
 $$ \|uv\|_{H^{l-1,-\half+2\epsilon}} \lesssim \|u\|_{H^{s+1,\half+\epsilon}} \|v\|_{H^{s-1,\half+\epsilon}} \, , $$
 which requires by Cor. \ref{Cor.2}: $1-l+s+1+s-1 >\frac{3}{4} \, \Leftrightarrow \, 2s-l > -\frac{1}{4}$ and $1-l+s+1+s-1 > 1-(\half+2s) \, \Leftrightarrow \, 4s-l > -\half$ . \\
 {\bf Proof of (\ref{32}):}
 The desired estimate follows from Cor. \ref{Cor.1} as follows:
 \begin{align*}
 \| u \Lambda^{-1}(vw)\|_{H^{s-1,-\half+2\epsilon}} & \lesssim \|u\|_{H^{l+1,\half+\epsilon}} \| \Lambda^{-1}(vw)\|_{H^{s-\half,0}} \\
 & \lesssim \|u\|_{H^{l+1,\half+\epsilon}} \|v\|_{H^{s,\half+\epsilon}} \|w\|_{H^{s-1,\half+\epsilon}} \, , 
 \end{align*}
 where we used $l > -\half$ and $s > \half$ . \\
 {\bf Proof of (\ref{33}):}
 By Cor. \ref{Cor.1} we obtain
  \begin{align*}
 \| u \Lambda^{-1}(vw)\|_{H^{s-1,-\half+2\epsilon}} & \lesssim \|u\|_{H^{l,\half+\epsilon}} \| \Lambda^{-1}(vw)\|_{H^{s+\half,0}} \\
 & \lesssim \|u\|_{H^{l,\half+\epsilon}} \|v\|_{H^{s,\half+\epsilon}} \|w\|_{H^{s,\half+\epsilon}} \, , 
 \end{align*}
 because $ l \ge s-1$ ,  $l > -\half$ and $s > \half$ . \\
{\bf Proof of (\ref{35}) and (\ref{36}):}
By Cor. \ref{Cor.1} we obtain for $s > \half$ :
$$\|uvw\|_{H^{s-1,-\half+2\epsilon}} \lesssim \|u\|_{H^{s,\half+\epsilon}} \|vw\|_{H^{s-\half,0}} 
  \lesssim \|u\|_{H^{s,\half+\epsilon}} \|v\|_{H^{s,\half+\epsilon}} \|w\|_{H^{s,\half+\epsilon}} \, . $$
 {\bf Proof of (\ref{34}):}
  Using Prop. \ref{Prop.7.1}  we obtain for $s> \half$ :
  \begin{align*}
  \| \Lambda^{-1}(uv) \Lambda^{-1}(wz)\|_{H^{s-1,-\half+2\epsilon}} & \lesssim \|\Lambda^{-1}(uv)\|_{H^{s+,3\epsilon}} \| \Lambda^{-1} (wz)\|_{H^{0,0}} \\
  & \lesssim \|u\|_{H^{s,\half+\epsilon}} \|v\|_{H^{s,\half+\epsilon}} \|w\|_{H^{s,\half+\epsilon}} \|z\|_{H^{s-1,\half+\epsilon}} \, .
  \end{align*}
  {\bf Proof of (\ref{37}):}
  We obtain
  \begin{align*}
  \|uvwz\|_{H^{l-1,-\half+2\epsilon}} & \lesssim \|uv\|_{H^{s-\frac{1}{4},2\epsilon}} \|wz\|_{H^{s-\frac{1}{4},2\epsilon}} \\& \lesssim \|u\|_{H^{s,\half+\epsilon}} \|v\|_{H^{s,\half+\epsilon}} \|w\|_{H^{s,\half+\epsilon}} \|z\|_{H^{s,\half+\epsilon}}
  \end{align*}
by Prop. \ref{Prop.7.1}, where we used $1-l+2s-\half > 1 \, \Leftrightarrow \, 2s-l > -\half$ , $s \ge l-\frac{3}{4}$ and $s >\frac{3}{4}$ .

\section{Proof of Theorem \ref{Theorem0.3}:}
In section 6 we proved the estimates (\ref{21}) - (\ref{37}) for
$$r=1+ \, , \, s=\frac{3}{2}+\epsilon \, , \, l=\half+\epsilon $$
for any $\epsilon>0$ and $b = \frac{1}{r}+$  and in section 7 for
$$r=2 \, , \, s=\frac{3}{4}+\epsilon \, , \, l=-\frac{1}{4}+\epsilon $$
for any $\epsilon >0$ and $ b=\half+$ . By multilinear interpolation this implies the  validity of these estimates for
$$1 < r \le 2 \, , \, s=\frac{3}{2r}+\epsilon \, , \, l=\frac{3}{2r}-1+\epsilon \, , \,  b=\frac{1}{r}+ \, $$
for any $\epsilon > 0$ .
An application of Theorem \ref{Theorem1.1} implies Theorem \ref{Theorem0.3}.

\end{document}